\documentclass[DIV=classic,a4paper,10pt]{myartWithThanks}

\KOMAoptions{DIV=last}
\usepackage{atbegshi}% http://ctan.org/pkg/atbegshi
\AtBeginDocument{\AtBeginShipoutNext{\AtBeginShipoutDiscard}}

\usepackage{caption}
\usepackage{subcaption}

\newcommand{\norm}[1]{\left\Vert#1\right\Vert}
\newcommand{\abs}[1]{\left\vert#1\right\vert}

\DeclareMathOperator*{\esssup}{ess\,sup}

\usepackage{accents}

\newcommand{\diff}{\,\mathrm{d}}
\newcommand{\diffns}{\mathrm{d}}
\usepackage[dvipsnames]{xcolor}

\begin{document}

\title{Strong solutions of some one-dimensional SDEs with random and unbounded drifts}\thanks{We thank an associate editor and the anonymous referees for their feedback and suggestions, one of which led to Theorem \ref{thm:flow}. We also thank Daniel Ocone for his comments.}

\author[]{Olivier Menoukeu-Pamen}
\author[]{Ludovic Tangpi}

\abstract{ 
In this paper, we are interested in the following one dimensional forward stochastic differential equation (SDE)
\[
\diffns X_{t}=b(t,X_{t},\omega)\diffns t +\sigma\diffns B_{t},\,\,\,0\leq t\leq T,\,\,\,\text{ }X_{0}=\,x\in \mathbb{R},
\]
where the driving noise $B_{t}$ is a $d$-dimensional Brownian motion. The drift coefficient $b:[0,T] \times\Omega\times \mathbb{R}\longrightarrow
\mathbb{R}$ is Borel measurable and can be decomposed into a deterministic and a random part, i.e., $b(t,x,\omega) = b_1(t,x) + b_2(t,x,\omega)$. 
Assuming that $b_1$ is of spacial linear growth and $b_2$ satisfies some integrability conditions, we obtain the existence and uniqueness of a strong solution.
The method we use is purely probabilitic and relies on Malliavin calculus. 
As byproducts, we obtain Malliavin differentiability of the solutions, provide an explicit representation for the Malliavin derivative and prove existence of weighted Sobolev differentiable flows.
}

\date{\today}
\keyAMSClassification{60G15, 60G60, 60H07, 60H10}
\keyWords{Malliavin calculus, random drift, measurable drift, compactness criterion, explicit representation, Sobolev differentiable flow.}
\maketitleludo
\setcounter{page}{1} %

%\tableofcontents

\section{Introduction}
The first main result of the present paper concerns wellposedness of a class of stochastic differential equations of the form
\begin{equation}
\label{eq:sde into}
	dX_t = \left(b_1(t, X_t) + b_2(t,X_t,\omega)\right)\,dt + \sigma\diffns B_t, \quad 0\le t\le T,\,\, X_0=x\in \mathbb{R}
\end{equation}
when the drift coefficient $b_1:[0,T]\times  \mathbb{R}\to \mathbb{R}$ is merely measurable and of linear growth in the second variable, $b_2$ is bounded (in $x$), sufficiently smooth and has bounded derivative (see the assumptions in Section \ref{sectionmainres11}) and $\sigma\in \mathbb{R}^d$.
The driving noise $B$ is the canonical process $B_t(\omega) = \omega_t$ on the canonical space $\Omega := C([0,T],\mathbb{R}^d)$ equipped with the Wiener measure $P$ and the completion $({\cal F})_{t\in [0,T]}$ of the natural filtration of $B$.

Since the work of \citet{Ito46}, it is well known that the SDE \eqref{eq:sde into} admits a unique strong solution when the drift $b$ is globally Lipschitz continuous and of linear growth.
That is, there exists a unique (up to  indistinguishability) square integrable process that is $\mathbb{F}$-adapted and satisfies \eqref{eq:sde into}.
SDEs are widely applied in stochastic control, in physics, and as a modeling tool, in a number of applied sciences including biology, finance and engineering.
Often, the Lipschitz continuity condition is too stringent, as for instance in modeling of switching systems (see e.g. \citet{Del14} and \citet{Hei-Laks94}) or in models of interacting finite (or infinite) particle systems (see e.g. \citet{Kond-Kons-Roeck04} and \citet{Alb-Kon-Roec03}) where the drift $b$ is typically discontinuous.
While existence of weak solutions of \eqref{eq:sde into} is a direct consequence of Girsanov theory, the construction of strong solutions is usually a delicate matter.
Note in addition that in the above mentioned applications existence of a solution $X$ as function of the driving noise (i.e. strong solution) is crucial.

Strong solutions of SDEs with rough coefficients have been extensively studied in the past decades, starting with the seminal works by \citet{Zvon74} and \citet{Ver79} and including other important contributions e.g. by \citet{GyK96}, \citet{GyM01}; \citet{KR05} or \citet{FreFlan13}.
These works eventually build on the analysis of the Kolmogorov partial differential equation associated to the SDE or on pathwise uniqueness arguments and benefit from the Yamada-Watanabe theorem.  

Let us further refer to works by \citet{Fang-Zhang05,FIZ07} and more rencently \citet{Cham-Jabin17} on uniqueness of SDEs.
See also \citet{Da07} for a path by path uniqueness result.
A purely probabilistic approach, initiated by \citet{Pro07} and \citet{MBP10}, and further developed by \citet{MMNPZ13} rather uses the Malliavin calculus of variations and white noise analysis for the construction of solutions (see also \citet{BMBPD17}).
This method does not rely on pathwise uniqueness arguments, but rather derives it as a consequence of uniqueness in law and strong existence.

As a common feature in the aforementioned works, the drift coefficient $b$ is assumed  deterministic (i.e. not depending on $\omega$).
This is due to the need to guarantee a Markovian property of the solution, which is paramount for the success of the PDE methods (in finite dimension). 
As suggested by an anonymous referee, let us mention however that it seems conceivable that, to some extend, the PDE methods  could work when the random part $b_2$ of the drift is seen as a forcing in the equation, but this remains an open question.
Random drift also constitute a clear impediment to the success of the probabilistic method 
due to the "integration by parts estimates" used in several steps of the proofs.
Regarding the growth of the drift coefficient, let us mention that to the best of our knowledge, the only works considering SDEs with discontinuous drifts and with linear growth are the articles by \citet{EnSc89}, \citet{Nil92} and \citet{MeMo17}.
Unbounded drifts are also treated under $L^q_t(L^p_x)$-type conditions, see e.g. \cite{KR05,FeFlan13,FGP10}. 
All these works consider deterministic coefficients.

In this paper we consider SDEs with coefficients $b$ of the form
\begin{equation}
\label{eq:b inro}
	b(t,x,\omega) = b_1(t,x) + b_2(t,x,\omega)
\end{equation}
for some random non-anticipating stochastic function $b_2$, with $b_1$ a Borel function of spacial linear growth.
In particular, $b_2$ is possibly path-dependent.
When $b_1$ is the gradient of a given function, such SDEs can be seen as dynamics of a diffusion in a random potential see e.g. \citet{Kond-Kons-Roeck04}.

SDEs with random coefficients when the drift coefficient does not have the special structure \eqref{eq:b inro} have been studied in the literature. 
For example, \citet{OcPa89} use the generalized It\^o-Ventzell formula for anticipating integrands to study a Stratonovich-type SDE, where the initial condition and drift coefficient are allowed to anticipate the future of driving Brownian motion. They show that the Stratonovich-type SDE with anticipating coefficients has a unique non exploding Malliavin differentiable solution.
They assume that the initial condition and drift coefficient are Malliavin smooth and the drift is further sublinear  with respect to the spatial coordinate and has derivatives of polynomial growth. Assuming that the drift coefficient satisfies a stochastic Lipschitz condition, \citet{HLN97} show existence and uniqueness of a class of SDE with random coefficients. 
They do not prove Malliavin differentiability of the solution in their work.

Our method draws from the Malliavin calculus approach of \citet{Pro07} and \citet{MMNPZ13} but avoids the use of white noise analysis.
In particular, we prove existence and uniqueness of a strong solution and further derive Malliavin differentiability and non-explosion of the solution. The main difficulties in deriving Malliavin smoothness of the solution to the SDE comes from the fact that we do not require any spatial smoothness of the coefficient.  
The estimates used to derive Malliavin differentiability of solutions allow to further obtain existence of a stochastic flow of dynamical systems driven by \eqref{eq:sde into}, and this hints at applications to new stochastic transport equations as in the work by \citet{MNP2015} (see also \citet{FGP10}, \citet{FeFlan13} and \citet{MenouA17}).

% In the present setting, one could view the SDE \eqref{eq:sde into} as an SDE of the form 
% $$
% 	\diffns X_t = b(t,X_t)\diffns t + \diffns A_t 
% $$
% where the driver $A$ is the semimartingale given by $\diffns A_t=b_2(t,X_t,\omega)\diffns t+\sigma\cdot \diffns B_t$. 
% However, the techniques used in \citet{MeMo17} cannot directly be applied since the bound of the Malliavin derivative (see  Proposition 4.10. in \citet{MeMo17}) makes use of the Girsanov transform, the Gaussian density of Brownian motion in the crucial integration by parts estimates.

Let us now give a precise statement of the main results of the paper.

\subsection{Probabilistic setting}
Let $T \in (0,\infty)$ and $d \in \mathbb{N}$ be fixed and consider a probability space $(\Omega, {\cal F}, P)$ equipped with the completed filtration $({\cal F}_t)_{t\in[0,T]}$ of a $d$-dimensional Brownian motion $B$.
Throughout the paper, the product $\Omega \times [0,T]$ is endowed with the predictable $\sigma$-algebra. Subsets of $\mathbb{R}^k$, $k\in\mathbb{N}$, are always endowed with the Borel $\sigma$-algebra induced by the Euclidean norm $|\cdot|$. The interval $[0,T]$ is equipped with the Lebesgue measure. Unless otherwise stated, all equalities and inequalities between random variables and processes will be understood in the $P$-almost sure and $P\otimes \diffns t$-almost sure sense, respectively.
For $p \in [1, \infty]$ and $k\in\mathbb{N}$, denote by ${\cal S}^p(\mathbb{R}^k)$
the space of all adapted continuous processes $X$
with values in $\mathbb{R}^k$ such that
$\norm{X}_{{\cal S}^p(\mathbb{R}^k)}^p :=  E[(\sup\nolimits_{t \in [0,T]}\abs{X_t})^p] < \infty$,
and by ${\cal H}^p(\mathbb{R}^{k})$ the space of all predictable processes $Z$ with values in $\mathbb{R}^{k}$ such that $\norm{Z}_{{\cal H}^p(\mathbb{R}^{k})}^p := E[(\int_0^T\abs{Z_u}^2\diffns u)^{p/2}] < \infty$.

\subsection{Main results}\label{sectionmainres11}
In this section, we present the main results of the paper.
Refer to the beginning of Section \ref{sec:SDE} for details regarding Malliavin calculus.
Let us consider the following conditions
\begin{enumerate}[label = (\textsc{A1}), leftmargin = 30pt]
	\item It holds $b= b_1 + b_2$, where the function $b_1:[0,T]\times \mathbb{R}\to \mathbb{R}$ is Borel measurable and there is $k_1\ge 0$ such that for all $x \in \mathbb{R}$, $$|b_1(t,x)|\le k(1 + |x|).$$ 
	The function $b_2:[0,T]\times \mathbb{R}\times \Omega\to \mathbb{R}$ is adapted, such that
	\begin{enumerate}
	\item $b_2(t,\cdot,\omega) \in C^1(\mathbb{R})$ and there exists a random variable $M_2(\omega)\ge0$ such that \begin{equation}
		|\frac{\partial}{\partial x}b_2(t,x,\omega)|+|b_2(t,x,\omega) |\leq M_2(\omega) \text{ for all }(t,x,\omega) \in [0,T]\times \mathbb{R}\times \Omega,
	\end{equation}  
	with
	\begin{equation}
		\label{eq:expo moment b2}
			b_2^{\mathrm{exp}}:=E\left[e^{C|M_2(\omega)|^2}\right]<\infty
	\end{equation}
	and $C:=48T\max_{i}\frac{d\sigma_i^2}{(\sigma_1^2+\cdots +\sigma_d^2)^2} $
	%In the above the "'" denote the derivative with respect to the second variable and 
	The constant $C$ depends on $ k, d$ and $\sigma$ (see Lemma \ref{lemmaproexpr1} for its explicite form).
	\item The random variable $b_2(s,x,\cdot)$ is Malliavin differentiable for every $(s,x) \in[0,T]\times \mathbb{R}$ and there exists a process $\tilde{M}_2(t,\omega)$ such that its Malliavin derivative $D_tb_2(s,x,\omega)$ satisfies
$$
|D_tb_2(s,x,\omega)|\leq \tilde{M}_2(s,t,\omega) \quad P\otimes dt\text{-a.s., for all } (t,x) \in [0,T]\times \mathbb{R},
$$
with
	\begin{equation}
	\label{eq: power moment b2}
		b_2^{\mathrm{power}}:= \sup_{0\le s\le T}E\left[\left(\int_0^T|\tilde{M}_2(s,t,\omega) |^2\diffns t \right)^4 \right]<\infty
	\end{equation}
	and there exist constants $C, \alpha'>0$ such that 
	$$
	E[|D_tb_2(s,x,\omega)-D_{t'}b_2(s,x,\omega)|^4]\leq C|t'-t|^{\alpha}
	$$
		\end{enumerate}
			\label{a1}
\end{enumerate}
\begin{enumerate}[label = (\textsc{A2}), leftmargin = 30pt]
	\item $\sigma \in \mathbb{R}^d$ and $|\sigma|^2>0$.\label{a2}
\end{enumerate}
\begin{theorem}\label{thmainres1r}
	Assume that conditions \ref{a1}-\ref{a2} hold.
 	Then there exists a unique global strong solution $X \in {\cal S}^2(\mathbb{R})$ to the SDE 
 	\begin{equation}
		\diffns X_{t}=b(t, X_{t},\omega) \diffns t+ \sigma\diffns B_{t},\,\,\,0\leq t\leq T,\,\,\,\text{ }X_{0}=\,x\in \mathbb{R}.  \label{eqmain1r}
	\end{equation}
\end{theorem}
The proof is given in Sections \ref{sec:SDE} and \ref{sec:SDE.Malliavin}.
Under the conditions of Theorems \ref{thmainres1r}, we show that the unique strong solution of the SDE is Malliavin differentiable and has a Sobolev differentiable flow.
Namely, we have
\begin{theorem}\label{thm:Mall.diff.SDE}
	Assume that \ref{a1}-\ref{a2} hold.
	Let $X$ be the unique strong solution to the SDE \eqref{eqmain1r}.
	It holds $X_t \in {\cal D}^{1,2}(\mathbb{R})$ for all $t \in [0,T]$.
\end{theorem}
The proof of this result is given in Section \ref{sec:SDE.Malliavin} where moment bounds and an explicit representation of the Malliavin derivative is also provided.
Denote by $X^{s,x}$ the solution of the SDE \eqref{eq:sde into} with initial condition $X_s=x$.
As is well-known, the solution $X^{s,x}$ may not belong to the Sobolev space $W^{1,p}(\mathbb{R}, dx)$, $p>1$.
Thus, following the intuition of \citet{MNP2015} we will show that $X^{s,x}$ belongs to a weighted Sobolev space.

Let $w:\mathbb{R}\to (0,\infty)$ be a Borel-measurable (weight) function such that
\begin{equation*}
	\int_\mathbb{R}e^{c|x|^2}w(x)\diffns x<\infty
\end{equation*}
for every $c\ge0$.
We denote by $W^{1,p}(\mathbb{R}, w)$ the weighted Sobolev space of functions $u:\mathbb{R}\to\mathbb{R}$ such that, it holds
\begin{equation*}
	||u||_{1, p, w} := \left(\int_\mathbb{R}|u(x)|^pw(x)\diffns x \right)^{1/p} + \left(\int_\mathbb{R}|u'(x)|^pw(x)\diffns x \right)^{1/p}<\infty,
\end{equation*}
where $u'$ is the weak derivative of $u$.
\begin{theorem}
\label{thm:flow}
	Assume that \ref{a1}-\ref{a2} hold.
	Let $X^{x}$ be the unique solution of \eqref{eq:sde into}.
	Let $T$ be small enough.
	Then for every $p
	\ge2$, the map $x \mapsto X^{x}_t$ belongs to $L^2(\Omega, W^{1,p}(\mathbb{R}, w))$.
\end{theorem}
It would be desirable to have multi-dimensional versions of Theorems \ref{thmainres1r} \ref{thm:Mall.diff.SDE} and \ref{thm:flow}.
The main obstacles to the extension of the method presented in this paper to the multi-dimensional case are presented in Remark \ref{rem:muldi-d}.
Let us give some examples of drift coefficients satisfying condition \ref{a1}.
{

\begin{example}
	The example of a random drift term of the form $b_1(t,x) + \varphi(t,x,B_t)$, where $\varphi:[0,T]\times\mathbb{R}\times\mathbb{R}^d\to \mathbb{R}$ is a Lipschitz continuous functions (in the second and third variables) seems not to be covered by the existing literature.
	It is consistent with our assumptions since the Malliavin derivative of $\varphi(t,x,B_t)$ is bounded, and the exponential moment condition \eqref{eq:expo moment b2} is satisfied, at least for $T$ small enough, or for arbitrary $T$ when $\varphi$ is bounded.

	A more general example is the path dependent drift case  $b(t,\omega, x):=b_1(t,x) + \varphi(t,x,B_{0:t})$, where $B_{0:t}$ denotes the path of $B$ up to $t$, and $\varphi:[0,T]\times \mathbb{R}\times C([0,T],\mathbb{R}^d)\to \mathbb{R}$ is bounded and Lipschitz continuous.
	It follows e.g. from \citet[Proposition 3.2]{Che-Nam} that $\varphi(t,x,B_{0:t})$ has bounded Malliavin derivatives for all $t$.
\end{example}

% \begin{example}
% 	Let $\alpha(t,x)$ be a smooth and Lipschitz function such that $\alpha \in L^2(dx)$ and is bounded, and put $b_2(t,x,\omega):=\int_0^t\alpha_s(x)\diffns B_s$ satisfies the moment condition \eqref{eq:expo moment b2} is satisfied (since $\alpha$ is deterministic). %, and $b_2$ is Malliavin differentiable.
% 	It follows by \citet[Proposition 1.3.4]{Nua06} that $b_2$ is Malliavin differentiable, and
% 	\begin{equation}
% 		D_sb_2(t,x,\omega) = \alpha_s(x)1_{\{s\le t \}}. % + \int_s^tD_s\alpha_r\diffns B_r.
% 	\end{equation}
% 	Thus, by Burkholder-Davis-Gundy inequality, it holds
% 	$$
% 		\sup_{0\le s\le T}E\left[\left(\int_0^T|D_sb_2(t,x,\omega)|^2\diffns t\right)^4 \right]\le \|\alpha\|_{L^2(dt)}^2 %+ \sup_{0\le s\le T}\|D_s\alpha \|^4_{{\cal H}^4(\mathbb{R}^d)}
% 		<\infty,
% 	$$
% 	which shows that $b^{\textrm{power}}_2<\infty$.
% \end{example}

The remainder of the paper is structured as follows:
The next section is mainly dedicated to the proof of Theorem \ref{thmainres1r}. 
As a byproduct of our method, we obtain Malliavin differentiability of the solution.
In addition, we derive various results concerning the Malliavin derivative of the solution, including moment estimates and a representation in terms of the space-time local time integral.
In the appendix we present a few auxiliary results to make the paper self-contained.

\section{Existence and Uniqueness}
\label{sec:SDE}

\subsection{Some notation}
In this section, we prove existence and uniqueness of strong solutions for SDEs.
Since Malliavin calculus will play an important role in our arguments, we briefly introduce the spaces of Malliavin differentiable random variables and stochastic processes
${\cal D}^{1,p}(\mathbb{R}^k)$ and ${\cal L}^{1,p}_a(\mathbb{R}^k)$, $p\ge 1$. For a thorough treatment of the theory of Malliavin calculus we refer to \citet{Nua06}.
Let ${\cal M}$ be the class of smooth random variables $\xi=(\xi^1,\dots,\xi^l)$ of the form
\begin{equation*}
	\xi^i = \varphi^i\left(\int_0^T h^{i1}_s\diffns W_s, \dots, \int_0^T h^{in}_s\diffns W_s \right),
\end{equation*}
where $\varphi^i$ is in the space $C_{\text{poly}}^\infty(\mathbb{R}^{n};\mathbb{R})$  of infinitely continuously differentiable functions whose partial derivatives have polynomial growth, $h^{i1}, \dots, h^{in} \in L^2([0,T]; \mathbb{R}^d)$ and $n\ge 1$.
For every $\xi$ in ${\cal M}$ let the operator $D = (D^1, \dots, D^d):{\cal M}\to L^2(\Omega\times [0,T];\mathbb{R}^d)$ be  given by
\begin{equation*}
	D_t\xi^i := \sum_{j=1}^n\frac{\partial \varphi^i}{\partial x_{j}}\left(\int_0^T h^{i1}_s\diffns W_s, \dots, \int_0^T h^{in}_s\diffns W_s \right)h^{ij}_t, \quad 0\le t\le T,\,\, 1\le i\le l,
\end{equation*}
and the norm $	\norm{\xi}_{1,p} := (E[\abs{\xi}^p + \int_0^T\abs{D_t\xi}^p\diffns t  ] )^{1/p}$.
As shown in \citet{Nua06}, the operator $D$ extends to the closure ${\cal D}^{1,p}(\mathbb{R}^l)$ of the set ${\cal M}$ with respect to the norm $\norm{\cdot}_{1,p}$.
A random variable $\xi$ is Malliavin differentiable if $\xi \in {\cal D}^{1,p}(\mathbb{R}^l)$ and we denote by $D_t\xi$ its Malliavin derivative.
Denote by ${\cal L}^{1,p}_a(\mathbb{R}^{l})$ the space of processes $Y \in {\cal H}^2(\mathbb{R}^{l})$ such that
$Y_t \in {\cal D}^{1,p}(\mathbb{R}^{l})$ for all $t \in [0,T]$, the process $DY_t$ admits a square integrable progressively measurable version and
\begin{equation*}
	\norm{Y}_{{\cal L}^{1,p}_a(\mathbb{R}^l)}^p := \norm{Y}_{{\cal H}^p(\mathbb{R}^l)} + E\left[\int_0^T\int_0^T\abs{D_r Y_t}^p\diffns r\diffns t \right] < \infty.
\end{equation*}

\subsection{Proof of Theorem \ref{thmainres1r}}
In the whole of this section, we assume that conditions \ref{a1} and \ref{a2} are satisfied. 
The proof of the Theorem \ref{thmainres1r} is given in 5 steps.
In the first step, 
we show that there exists a process $X^x$ satisfying the SDE \eqref{eqmain1r} in the weak sense. 
That is, there is a Brownian motion $\tilde B$ such that $(X^x_t, \tilde B_t)$ is a weak solution to the SDE \eqref{eqmain1r}. Note however that the solution might not be adapted to the filtration $(\mathcal{F}_t)_{t\in [0, T]}$. 
Let us mention that if $X_t$ is adapted to that filtration then $X_t$ has an explicit representation as a function of $B_t$ (see for example \cite{LP04, MBpr04}) and for any other stochastic basis  $(\tilde{\Omega}, \tilde{\mathcal{F}},(\tilde{\mathcal{F}}_t)_{t \in [0,T]}, \tilde{P},\tilde{B})$, the same representation holds with $\tilde{B}_t$ instead of $B_t$ and thus $X_t$ is $(\tilde{\mathcal{F}}_t)_{t\in [0,T]}$-adapted. The latter indicates that $X_t$ is a strong solution of \eqref{eqmain1r}.
 	
In the second step, for $T$ small, given a sequence $b_n:=b_{1,n}+b_2$ such that $b_{1,n}: [0,T] \times \mathbb{R} \rightarrow \mathbb{R}$, $n\ge 1$ are smooth coefficients with compact support
 and converging a.e. to $b_1$, we show using relative compactness (see Lemma \ref{lemmainres1r}) that for each $0 \leq t \leq T$ the sequence of corresponding strong solutions $(X^{x,n}_{t})_{n \geq 1}$, of the SDEs
\begin{align}\label{eqmainr}
 	\diffns X^{x,n}_t=b_n(t,X^{x,n}_t,\omega)\diffns t + \sigma\cdot \diffns B_t, \,\,0\leq t\leq T,\,\,\, X^{x,n}_0=x  \in \mathbb{R}, \, n \geq 1,
\end{align}
is relatively compact in $L^2(P;\mathbb{R})$. Let us mention that existence of a unique strong solution to the SDE \eqref{eqmainr} is guaranteed by \cite[Theorem 1.1]{OcPa89}, see also \cite{KaShr88}.

In step 3, we show that for each $0 \leq t \leq T$ the above sequence $(X^{x,n}_{t})_{n \geq 1}$ converges weakly to $E\Big[X^x_t|\mathcal{F}_t\Big]$ in the space $L^2(\Omega,\mathcal{F}_t,P)$. This with step 2 allow to deduce that $(X^{x,n}_{t})_{n \geq 1}$ converges strongly to $E\Big[X^x_t|\mathcal{F}_t\Big]$ in the space $L^2(\Omega,\mathcal{F}_t,P)$.  We also obtain from step 2 that $E\Big[X^x_t|\mathcal{F}_t\Big]$ is Malliavin smooth, see Subsection \ref{sec:SDE.Malliavin}.

In step 4, we prove that $X^x_t $ is $\mathcal{F}_t$-measurable by showing that 
 $E\Big[X^x_t|\mathcal{F}_t\Big]=X^x_t $. The proof is completed by showing uniqueness. 
 
In the last step, we use a pasting argument to show that the result holds for all $T>0$.
In fact, the linear growth assumption on $b_1$ and integrability assumption on $b_2$ ensure by the use of Gronwall's lemma that if the solution exists on a small interval then it does not explode. Hence the main task in this step is to show that $E[|D_sX^{x,n}_{t}|^2]\leq C$, uniformly in $n$ for $0\leq s\leq t\leq T$.

\paragraph{2.2.1. Weak existence.}
The following result can be seen as a slight generalization of a result by V.E. Bene\v{s}, compare \cite{Ben71, KleLip12}. 
Therein (and throughout the paper) we denote by ${\cal E}(\int q\diffns B)$ the Dol\'ean-Dade exponential
\begin{equation*}
	{\cal E}\left(\int q\diffns B\right)_t := \exp\left(\int_0^t q_u\diffns B_u - \frac 12 \int_0^t |q_u|^2\diffns u \right)
\end{equation*}
for a given progressive process $q$ such that $\int_0^T|q_u|^2\,du<\infty$.
\begin{lemma}\label{bengen}
	Let $u$ be given by 
 	\begin{align}\label{eqgirs1}
 		u_i=\frac{\sigma_i}{\sigma_1^2+\ldots +\sigma_d^2 }b.
\end{align}
Then the process $Z:=\mathcal{E}\left( \int u(r,\sigma\cdot B_r,\omega)\diffns B_r\right)$ is a martingale.
\end{lemma}
\begin{proof}
	The proof follows from that of \cite[Theorem 2.1]{KleLip12}. We know that the process $Z_t$ is a non negative local martingale and thus a supermartingale such that $E[Z_t]\leq 1$ for any $0\leq t\leq T$. 
	Using the same argument as in the proof of \cite[Theorem 2.1]{KleLip12}, one obtains the result by applying Gronwall's lemma provided that  $E\left[|M_2(\omega)|^4\right]<\infty$. 
	The later is true by assumption. 
\end{proof}
 	
The next lemma ensures weak existence, it is a simple adaptation of \cite[Proposition 5.3.6]{KaShr88}. 
\begin{lemma} 
\label{lem:weak existence}
	The SDE \eqref{eqmain1r} admits a weak solution $X^x_t$. 
\end{lemma}
 	
\begin{proof}
 	Let $(\Omega, \mathcal{F},P)$ be a probability space on which a $d$-dimensional Brownian motion $\hat{B}$ is given, and set $X^x_t:=x+\sigma\cdot\hat{B}_t, \,0\leq t\leq T$. 
 	By \ref{a1}, it follows from Lemma \ref{bengen} (see also \cite{BuFo1993, Ben71, KleLip12}) that the process $\mathcal{E}\Big(\int u(r,X_r^x,\omega)\diffns \hat{B}_r\Big)$ defines an equivalent probability measure $Q$ given by
 	$$
 	\frac{\diffns Q}{\diffns P}:=\mathcal{E}\left(\int u(r,X_r^x,\omega)\diffns \hat{B}_r \right)_T.
 	$$
 	In addition, Girsanov's theorem asserts that 
 	$B_t= \hat B_t - \int_0^tu(r,X_r^x,\omega)\diffns r$
 	is a $Q$-Brownian motion.
 	Therefore, 
 	\begin{equation}
 	\label{eq:weak equation}
 		X^x_{t}=x+\int_0^t\sigma \cdot u(s, X^x_{s},\omega) \diffns s +\sigma\cdot  B_{t} \quad Q\text{-a.s.}  ,\,\,\,0\leq t\leq T,
 	\end{equation}
	showing that $(X^x, B)$ is a weak solution to the SDE \eqref{eqmain1r} on the probability space $(\Omega, \mathcal{F},Q)$. 
\end{proof}

\paragraph{2.2.2. Approximation and compactness.} 	
\label{par:approximation}
Let $b_n=b_{1,n}+b_2$ be such that $b_{1,n}: [0,T] \times \mathbb{R} \rightarrow \mathbb{R}$, $n\ge 1$ are smooth coefficients with compact support and converging a.e. to $b_1$.
Denote by $X^{x,n}_{t}$ the unique strong solution to the SDE \eqref{eqmainr} with drift $b_n$.
The following result is key to the compactness argument.

\begin{lemma} \label{lemmainres1r}
	If $T\in (0,\infty)$ is small enough, the strong solution $X^{x,n}_{t}$ of the SDE \eqref{eqmainr} satisfies
 	$$
 		E \left[ | D_t X^{x,n}_s - D_{t'} X^{x,n}_s |^2 \right] \leq {\cal C}(\|\tilde{b}_1\|_{\infty},|x|^2, b_2^{\text{power}}) |t -t'|^{\alpha}
 	$$ 
 		for all $0 \leq t' \leq t \leq T$ and some $\alpha = \alpha(s) > 0$.
 	Moreover, 
 	$$
 		\sup_{0 \leq t \leq T} E \left[ | D_t X^{x,n}_s |^2 \right] \leq {\cal C}(\|\tilde{b}_1\|_{\infty},|x|^2,b_2^{\text{power}}),
 	$$
 	where the function ${\cal C}(\cdot, \cdot, \cdot): [0, \infty)^3 \rightarrow [0, \infty)$ is continuous and increasing in each components, $b_2^{\text{power}}$ defined in \eqref{eq: power moment b2} and
 	\begin{equation}
 	\label{eq:def b infty}
 		\|\tilde{b}_1\|_{\infty}:= \displaystyle \esssup \left\{\frac{|b_1(t,z)|}{1+|z|}: t \in [0,T], z \in \mathbb{R} \right\}. 
 	\end{equation}
\end{lemma}

The combination of Lemma \ref{lemmainres1r} and Corollary \ref{compactcrit} yields the following result:
\begin{corollary} \label{maincor}
	For each $0 \leq t \leq T$, with $T$ small enough, the sequence $(X^{x,n}_{t})_{n \geq 1}$, is relatively compact in $L^2(P;\mathbb{R})$.
\end{corollary}

\begin{proof}[of Lemma \ref{lemmainres1r}]
	Since the Brownian motions are independent, applying the chain-rule formula for the Malliavin derivatives in the direction of the $i^{th}$ Brownian motion (see e.g. \cite{Nua06}) gives%
 	\begin{align}\label{malldifeqr}
 		D_t^iX^{x,n}_s &= \sigma_i+ \int_t^s D^i_tb_2(u,X^{x,n}_u,\omega)\diffns u+ \int_t^s \Big\{b_{1,n}'(u,X^{x,n}_u)+b_{2}'(u,X^{x,n}_u,\omega)\Big\}D_t^i X^{x,n}_u \diffns u, \,\,i=1,\ldots,d \,\, P\text{-a.s.}
 	\end{align}
 	for all $t \leq s \leq T$. Here $b_{1,n}'(t,x) := \frac{ \partial}{ \partial x} b_{1,n}(t,x) $ and $b_{2}'(t,x,\omega) := \frac{ \partial}{ \partial x} b_{2}(t,x,\omega) $ are  the spatial derivatives of $b_{1,n}$ and $b_2$, respectively. 
 	Solving \eqref{malldifeqr} explicitly gives
 	\begin{align} \label{MalliavinDerivativeEquationr}
 		D^i_tX^{x,n}_s &= e^{\int_t^s \{b_{1,n}'(u,X^{x,n}_u) +b_{2}'(u,X^{x,n}_u,\omega) \}\diffns u}\Big(\int_t^sD^i_tb_2(u,X^{x,n}_u,\omega)e^{-\int_t^u \{b_{1,n}'(r,X^{x,n}_r) +b_{2}'(r,X^{x,n}_r,\omega) \}\diffns r}\diffns u + \sigma_i \Big).
 	\end{align}
 			Let $ 0 \leq t' \leq t \leq s\leq T$. Using the above representation, we have
 			\begin{align}
 			& D^i_{t'} X^{x,n}_s - D^i_t X^{x,n}_s \notag\\
 			=&  e^{\int_{t'}^s\{b_{1,n}'(u,X^{x,n}_u) +b_{2}'(u,X^{x,n}_u,\omega) \} \diffns u}\Big(\int_{t'}^sD^i_{t'}b_2(u,\omega)e^{-\int_{t'}^u b_{2}'(r,X^{x,n}_r,\omega)\diffns r} e^{-\int_{t'}^u b_{1,n}'(r,X^{x,n}_r) \diffns r}\diffns u + \sigma_i\Big)\notag\\
 			& -e^{\int_t^s \{b_{1,n}'(u,X^{x,n}_u) +b_{2}'(u,X^{x,n}_u,\omega) \} \diffns u}\left(\int_t^sD^i_{t}b_2(u,X^{x,n}_u,\omega)e^{-\int_{t}^u b_{2}'(r,X^{x,n}_r,\omega)\diffns r}e^{-\int_t^u b_{1,n}'(r,X^{x,n}_r) \diffns r}\diffns u + \sigma_i\right)\notag\\
 			=&\sigma_i e^{\int_{t}^s\{b_{1,n}'(u,X^{x,n}_u) +b_{2}'(u,X^{x,n}_u,\omega) \} \diffns u}\left(e^{\int_{t'}^t\{b_{1,n}'(u,X^{x,n}_u) +b_{2}'(u,X^{x,n}_u,\omega) \} \diffns u}-1\right)\notag\\
 			& +\int_{t'}^sD^i_{t'}b_2(u,X^{x,n}_u,\omega)e^{-\int_{s}^u b_{2}'(r,X^{x,n}_r,\omega)\diffns r}e^{-\int_{s}^u b_{1,n}'(r,X^{x,n}_r) \diffns r}\diffns u - \int_t^sD^i_{t}b_2(u,\omega)e^{-\int_{s}^u b_{2}'(r,X^{x,n}_r,\omega)\diffns r}e^{-\int_s^u b_{1,n}'(r,X^{x,n}_r) \diffns r}\diffns u  \notag
\end{align}
\begin{align}
\label{eqmalderpro111r}
 			=& \sigma_i e^{\int_{t}^s\{b_{1,n}'(u,X^{x,n}_u) +b_{2}'(u,X^{x,n}_u,\omega) \} \diffns u}\left(e^{\int_{t'}^t\{b_{1,n}'(u,X^{x,n}_u) +b_{2}'(u,X^{x,n}_u,\omega) \} \diffns u}-1\right)\notag\\
 			&+\int_{t'}^tD^i_{t'}b_2(u,X_u^{n,x},\omega)e^{-\int_{s}^u b_{2}'(r,X^{x,n}_r,\omega)\diffns r}e^{-\int_{s}^u b_{1,n}'(r,X^{x,n}_r) \diffns r}\diffns u\notag\\
 			& +\int_{t}^s\left(D^i_{t'}b_2(u,X_u^{n,x},\omega)-D^i_{t}b_2(u,X_u^{n,x},\omega)\right)e^{-\int_{s}^u b_{2}'(r,X^{x,n}_r,\omega)\diffns r}e^{-\int_{s}^u b_{1,n}'(r,X^{x,n}_r) \diffns r}\diffns u  \notag\\
 			&=  I_1+I_2+I_3.
	\end{align}
	Set 
	\begin{equation}
	\label{eqgirs1 n}
		u_{i,n}=\frac{\sigma_i}{\sigma_1^2+\cdots + \sigma_d^2 }b_n.
	\end{equation}
	Then using H\"older inequality repeatedly, we have
	\begin{align}\label{mallderI1r}
		E[I_1^2]=&\sigma_i^2 E\Big[ e^{2\int_{t}^s b_{2}'(u,X^{x,n}_u,\omega) \diffns u} e^{2\int_{t}^s b_{1,n}'(u,X^{x,n}_u) \diffns u}\Big(e^{\int_{t'}^t b_{1,n}'(u,X^{x,n}_u) \diffns u}e^{\int_{t'}^t b_{2}'(u,X^{x,n}_u,\omega)\diffns u}-1\Big)^2\Big]\notag\\
		\le& \sigma_i^2 E\Big[ e^{6\int_{t}^s b_{2}'(u,X^{x,n}_u,\omega) \diffns u}\Big]^{\frac{1}{3}} E\Big[ e^{6\int_{t}^s b_{1,n}'(u,X^{x,n}_u) \diffns u}\Big]^{\frac{1}{3}}E\Big[ \Big(e^{\int_{t'}^t b_{1,n}'(u,X^{x,n}_u) \diffns u}e^{\int_{t'}^t b_{2}'(u,X^{x,n}_u,\omega) \diffns u}-1\Big)^6\Big]^{\frac{1}{3}}\notag\\
	=&	\sigma_i^2 E\Big[ e^{6\int_{t}^s b_{2}'(u,X^{x,n}_u,\omega) \diffns u}\Big]^{\frac{1}{3}} E\Big[ \mathcal{E}\Big(\int_0^Tu_n(r,x+\sigma\cdot B_r,\omega)\diffns B_r\Big)e^{6\int_{t}^s b_{1,n}'(u,x+\sigma\cdot B_u) \diffns u}\Big]^{\frac{1}{3}}\notag\\
	& \times E\Big[ \Big(e^{\int_{t'}^t b_{1,n}'(u,X^{x,n}_u) \diffns u}e^{\int_{t'}^t b_{2}'(u,X^{x,n}_u,\omega) \diffns u}-1\Big)^6\Big]^{\frac{1}{3}}\notag\\
		\leq & C  \sigma_i^2 E\Big[ e^{2\sum_{i=1}^d\int_0^Tu_{i,n}(r,x+\sigma\cdot B_r,\omega)\diffns B_r^i-2\sum_{i=1}^d\int_0^Tu_{i,n}^2(r,x+\sigma\cdot B_r,\omega)\diffns r}\Big]^{\frac{1}{6}}E\Big[e^{4\sum_{i=1}^d\int_0^Tu_{i,n}^2(r,x+\sigma\cdot B_r,\omega)\diffns r}\Big]^{\frac{1}{12}}\notag\\
		&\quad \times E\Big[e^{24\int_{t}^s b_{1,n}'(u,x+\sigma\cdot B_u) \diffns u}\Big]^{\frac{1}{12}} E\Big[ \Big(e^{\int_{t'}^t b_{1,n}'(u,X^{x,n}_u) \diffns u}e^{\int_{t'}^t b_{2}'(u,X^{x,n}_u,\omega) \diffns u}-1\Big)^6\Big]^{\frac{1}{3}}\notag\\
	=&	I_{1,1}^{\frac{1}{6}} +I_{1,2}^{\frac{1}{12}}+I_{1,3}^{\frac{1}{12}}+I_{1,4}^{\frac{1}{3}}.
	%	\leq &  E\Big[ e^{4\sum_{i=1}^d\int_0^Tu_{i,n}(r,\omega,x+\sigma\cdot B_r)\diffns B_r^i-8\sum_{i=1}^d\int_0^Tu_{i,n}^2(r,\omega,x+\sigma\cdot B_r)\diffns r}\Big]^{\frac{1}{4}} E\Big[e^{12\sum_{i=1}^d\int_0^Tu_{i,n}^2(r,\omega,x+\sigma\cdot B_r)\diffns r}\Big]^{\frac{1}{8}}\notag\\
		%&\quad\times E\Big[e^{16\int_{t}^s b_{1,n}'(u,x+\sigma\cdot B_u) \diffns u}\Big]^{\frac{1}{8}} E\Big[\Big(e^{\int_{t'}^t b_{1,n}'(u,x+\sigma\cdot B_u) \diffns u}-1\Big)^4\Big]^{\frac{1}{2}} \sigma_i^2.
	\end{align}
	It follows from the Girsanov theorem applied to the martingale $2\sum_{i=1}^d\int_0^\cdot u_{i,n}(r,x+\sigma\cdot B_r,\omega)\diffns B_r$ that the first term $I_{1,1}$ in \eqref{mallderI1r} is equal to one. 
	Next, we wish to use conditions on $b_n$ and thus $u_n$ to show that the second term is finite for $T$ small enough. 
 	Using H\"older inequality, we have
 	\begin{align}
 		E\left[e^{4\sum_{i=1}^d\int_0^Tu_{i,n}^2(r,x+\sigma\cdot B_r,\omega)\diffns r}\right]&\leq \prod_{i=1}^{d} E\left[e^{12d\int_0^Tu_{i,n}^2(r,x+\sigma\cdot B_r,\omega)\diffns r}\right]^{\frac{1}{d}}\notag
 		\leq  \prod_{i=1}^{d} E\left[e^{12d\int_0^T\frac{\sigma_i^2}{(\sigma_1^2+\ldots +\sigma_d^2)^2 }b_n^2(r,x+\sigma\cdot B_r,\omega)\diffns r}\right]^{\frac{1}{d}}\notag.
 	\end{align}
 	Let us focus on each component of the above product. Using the condition on $b_n$, H\"older inequality successively  and the independence of the Brownian motion, we get 
 	\begin{align}
 		E\left[e^{12d\int_0^T\frac{\sigma_i^2}{(\sigma_1^2+\cdots +\sigma_d^2)^2 }b_n^2(r,x+\sigma\cdot B_r,\omega)\diffns r}\right]	&\leq E\left[e^{24c_{d,\sigma}\int_0^T\left(k^2|1+x+\sigma\cdot B_r|^2 + |M_2(\omega)|^2\right)\diffns r}\right]\notag\\
 		&\leq E\left[e^{48c_{d,\sigma}\int_0^Tk^2(1+|x+\sigma\cdot B_r|)^2\diffns r} \right]^{\frac{1}{2}}\times E\left[e^{48c_{d,\sigma}T  |M_2(\omega)|^2}\right]^{\frac{1}{2}}\notag\\
 		&\leq  C e^{48c_{d,\sigma}k^2T(1+|x|)^2} E\left[e^{48c_{d,\sigma}k^2\int_0^T|\sigma\cdot B_r|^2\diffns r} \right]^{\frac{1}{2}}\notag\\
 		&\leq  C e^{48c_{d,\sigma}k^2T(1+|x|)^2} E\left[e^{24c_{d,\sigma}k^2T\sup_{0\leq t\leq T}|\sigma\cdot B_t|^2} \right]^{\frac{1}{2}}\notag
 \end{align}
 \begin{align}
 		&\leq  C e^{48c_{d,\sigma}k^2T(1+|x|)^2} E\left[e^{48c_{d,\sigma}k^2T\sum_{i=1}^d \sup_{0\leq t\leq T}|\sigma_i\cdot B^i_t|^2}\right]^{\frac{1}{2}}\notag\\
 		&\leq C_{T,d,\sigma,M_2} e^{48c_{d,\sigma}k^2T(1+|x|)^2} \prod_{i=1}^{d} E\left[e^{48c_{d,\sigma}k^2T\sup_{0\leq t\leq T}|\sigma_i\cdot B^i_t|^2}\right]^{\frac{1}{2}}.
 	\end{align}
 	In the above, $c_{d,\sigma}:=\frac{d\sigma_i^2}{(\sigma_1^2+\cdots +\sigma_d^2)^2} $. Now, using exponential expansion and the Doob maximal inequality, we have 
 	\begin{align}
 		E\left[e^{48c_{d,\sigma}k^2T\sup_{0\leq t\leq T}|\sigma_i\cdot B^i_t|^2} \right]&= 1 + \sum_{p=1}^\infty \frac{(48c_{d,\sigma}\sigma_i^2k^2T)^p}{p!} E\left[\sup_{0\leq t\leq T}|B_t|^{2p}\right]\notag\\
 		&\leq  1 + \sum_{p=1}^\infty \frac{(48k^2dT)^p}{p!} \left(\frac{2p}{2p-1}\right)^2\frac{(2p)!}{2^p\cdot p!}T^p.
 	\end{align}
 	The inequality comes from the fact that $\frac{d\sigma_i^2\sigma^2_i}{(\sigma_1^2+\cdots +\sigma_d^2)^2} \leq d$. 
 	Next, applying the ratio test to the series $\sum_{p} a_p$ with $a_p:=\frac{(48k_1^2dT^2)^p}{p!} \Big(\frac{2p}{2p-1}\Big)^2\frac{(2p)!}{2^p\cdot p!}$ for $p\geq 1$, one can easily show that the series converge for example for 
 	\begin{equation}
 	\label{eq:cond small time 1}
 		T\le T_1:= \frac{1}{4\sqrt{3}dk^2_1}.
 	\end{equation} 
 	Hence the second term in \eqref{mallderI1r} is finite for small T.
 		
 	Now, we turn to $I_{1,4}$ in \eqref{mallderI1r}.
 	Using power and exponential expansion, we get by linearity of the expectation and H\"older inequality
 		\begin{align}\label{eqmalcalI11r}
 &	E\Big[ \Big(e^{\int_{t'}^t b_{1,n}'(u,X^{x,n}_u) \diffns u}e^{\int_{t'}^t b_{2}'(u,X^{x,n}_u,\omega) \diffns u}-1\Big)^6\Big]\notag\\
 	= & E\Big[ e^{6\int_{t'}^t \{b_{1,n}'(u,X^{x,n}_u) + b_{2}'(u,X^{x,n}_u,\omega) \}\diffns u}-6e^{5\int_{t'}^t \{b_{1,n}'(u,X^{x,n}_u) + b_{2}'(u,X^{x,n}_u,\omega) \}\diffns u}+15e^{4\int_{t'}^t \{b_{1,n}'(u,X^{x,n}_u) + b_{2}'(u,X^{x,n}_u,\omega)\}\diffns u}\notag\\
 	&-20 e^{3\int_{t'}^t \{b_{1,n}'(u,X^{x,n}_u) + b_{2}'(u,X^{x,n}_u,\omega) \}\diffns u}+15e^{2\int_{t'}^t \{b_{1,n}'(u,X^{x,n}_u)+ b_{2}'(u,X^{x,n}_u,\omega) \}\diffns u}-6e^{\int_{t'}^t \{b_{1,n}'(u,X^{x,n}_u) + b_{2}'(u,X^{x,n}_u,\omega) \}\diffns u}+1\Big]\notag\\
 	=& E\Big[\sum_{q=1}^\infty \frac{\Big(6\int_{t'}^t \{b_{1,n}'(u,X^{x,n}_u) + b_{2}'(u,X^{x,n}_u,\omega) \}\diffns u\Big)^q}{q!}\Big]-6E\Big[\sum_{q=1}^\infty \frac{\Big(5\int_{t'}^t \{b_{1,n}'(u,X^{x,n}_u) + b_{2}'(u,X^{x,n}_u,\omega) \}\diffns u\Big)^q}{q!}\Big]\notag\\
 	&+15 E\Big[\sum_{q=1}^\infty \frac{\Big(4\int_{t'}^t \{b_{1,n}'(u,X^{x,n}_u) + b_{2}'(u,X^{x,n}_u,\omega) \}\diffns u\Big)^q}{q!}\Big]-20E\Big[\sum_{q=1}^\infty \frac{\Big(3\int_{t'}^t \{b_{1,n}'(u,X^{x,n}_u) + b_{2}'(u,X^{x,n}_u,\omega) \}\diffns u\Big)^q}{q!}\Big]\notag\\
 	&+15E\Big[\sum_{q=1}^\infty \frac{\Big(2\int_{t'}^t \{b_{1,n}'(u,X^{x,n}_u) + b_{2}'(u,X^{x,n}_u,\omega) \}\diffns u\Big)^q}{q!}\Big]-6E\Big[\sum_{q=1}^\infty \frac{\Big(\int_{t'}^t \{b_{1,n}'(u,X^{x,n}_u) + b_{2}'(u,X^{x,n}_u,\omega) \}\diffns u\Big)^q}{q!}\Big]\notag\\
 	\leq &  E\Big[\sum_{q=1}^\infty \frac{\Big|6\int_{t'}^t \{b_{1,n}'(u,X^{x,n}_u) + b_{2}'(u,X^{x,n}_u,\omega) \}\diffns u\Big|^q}{q!}\Big]+6E\Big[\sum_{q=1}^\infty \frac{\Big|5\int_{t'}^t \{b_{1,n}'(u,X^{x,n}_u) + b_{2}'(u,X^{x,n}_u,\omega) \}\diffns u\Big|^q}{q!}\Big]\notag\\
 	&+15 E\Big[\sum_{q=1}^\infty \frac{\Big|4\int_{t'}^t \{b_{1,n}'(u,X^{x,n}_u) + b_{2}'(u,X^{x,n}_u,\omega) \}\diffns u\Big|^q}{q!}\Big]+20E\Big[\sum_{q=1}^\infty \frac{\Big|3\int_{t'}^t \{b_{1,n}'(u,X^{x,n}_u) + b_{2}'(u,X^{x,n}_u,\omega) \}\diffns u\Big|^q}{q!}\Big]\notag\\
 	&+15E\Big[\sum_{q=1}^\infty \frac{\Big|2\int_{t'}^t \{b_{1,n}'(u,X^{x,n}_u) + b_{2}'(u,X^{x,n}_u,\omega) \}\diffns u\Big|^q}{q!}\Big]+6E\Big[\sum_{q=1}^\infty \frac{\Big|\int_{t'}^t \{b_{1,n}'(u,X^{x,n}_u) + b_{2}'(u,X^{x,n}_u,\omega) \}\diffns u\Big|^q}{q!}\Big]\notag\\
\leq  & J_1+J_2+J_3+J_4+J_5+J_6.
 	\end{align}
 	
 	Since $\sigma\cdot B_t \sim N(0, t\sum_{i=1}^d\sigma^2_i)$ and has independent increments, it follows from Proposition \ref{propmainEstimate} that each term in \eqref{eqmalcalI11r} is bounded by ${\cal C}(T,\|\tilde{b}_1\|_{\infty}, |x|)|t-t'|$, where ${\cal C}(\|\tilde{b}_1\|_{\infty}, |x|)$ is a continuous function depending on $\|\tilde{b}\|_{\infty}, x, \|\sigma\|^2$ and $T$.
 	More specifically, let us focus on $J_1$ only since the bounds for the other terms follow in a similar way.  Using dominated convergence theorem, H\"older inequality and Girsanov theorem, we have
 	\begin{align}\label{eqmalcalI12r}
 	J_1=&\sum_{q=1}^\infty \frac{E\Big[\Big|6\int_{t'}^t \{b_{1,n}'(u,X^{x,n}_u) + b_{2}'(u,X^{x,n}_u,\omega) \}\diffns u\Big|^q\Big]}{q!}\notag\\
 	\leq& \sum_{q=1}^\infty \frac{12^pE\Big[\Big|\int_{t'}^t b_{1,n}'(u,X^{x,n}_u)\diffns u\Big|^q\Big]}{q!} + \sum_{q=1}^\infty \frac{12^pE\Big[\Big|\int_{t'}^tb_{2}'(u,X^{x,n}_u,\omega) \diffns u\Big|^q\Big]}{q!}\notag\\
 	\leq& \sum_{q=1}^\infty \frac{12^pE\Big[\mathcal{E}\Big(\int_0^Tu_n(r,x+\sigma\cdot B_r,\omega)\diffns B_r\Big)\Big|\int_{t'}^t b_{1,n}'(u,x+\sigma\cdot B_u)\diffns u\Big|^q\Big]}{q!} + \sum_{q=1}^\infty \frac{12^p|t'-t|^qE\Big[C_T^p|M_{2}(\omega) |^q\Big]}{q!}\notag\\
 	\leq& E\Big[ e^{4\sum_{i=1}^d\int_0^Tu_{i,n}(r,\omega,x+\sigma\cdot B_r)\diffns B_r^i-8\sum_{i=1}^d\int_0^Tu_{i,n}^2(r,\omega,x+\sigma\cdot B_r)\diffns r}\Big]^{\frac{1}{4}}E\Big[e^{6\sum_{i=1}^d\int_0^Tu_{i,n}^2(r,\omega,x+\sigma\cdot B_r)\diffns r}\Big]^{\frac{1}{4}}\notag\\
 &	\times \sum_{q=1}^\infty \frac{12^pE\Big[\Big|\int_{t'}^t b_{1,n}'(u,x+\sigma\cdot B_u)\diffns u\Big|^{2q}\Big]^{\frac{1}{2}}}{q!} + |t'-t|^{\frac{1}{2}}\sum_{q=1}^\infty \frac{12^p|t'-t|^{q-\frac{1}{2}}E\Big[C_T^p|M_{2}(\omega) |^q\Big]}{q!}\notag\\
 \leq &  C \sum_{q=1}^\infty \frac{12^pE\Big[\Big|\int_{t'}^t b_{1,n}'(u,x+\sigma\cdot B_u)\diffns u\Big|^{2q}\Big]^{\frac{1}{2}}}{q!} +\frac{C}{\sqrt{T}}E\Big[\exp\{12C_TM_{2}(\omega) \}\Big]|t'-t|^{\frac{1}{2}},
 	\end{align}
 	where the first bound comes from Girsanov theorem applied to the martingale $2\sum_{i=1}^d\int_0^\cdot u_{i,n}(r,\omega,x+\sigma\cdot B_r)\diffns B_r$ and similar computations as in the case of $I_{1,2}$. Now by Proposition \ref{propmainEstimate},
 		\begin{align}\label{eqapenexpobound1} 
 		\sum_{q=1}^\infty \frac{E\Big[\Big(\int_{t'}^t \sqrt{12}b_{1,n}'(u,x+\sigma\cdot B_u) \diffns u\Big)^{2q}\Big]^{\frac{1}{2}}}{q!}
 		 \leq & \Big(\sum_{q=1}^{\infty}\frac{C_{\sigma,k}^q(1+|x|^{q})\sqrt{q!} (t-t')^{q/2}
 		 }{q!} \Big)\notag\\
 		 \leq &\Big(\sum_{q=1}^{\infty}\frac{C_{\sigma,k}^q(1+|x|^{q}) (t-t_0)^{\frac{q-1}{2}}
 		 }{\sqrt{q!}} \Big)|t-t'|^{1/2}
 		\end{align}	
 		for some positive constant $C_\sigma$. 

 	Multiplying the numerator and the denominator of each term in the series by $2^q$ and using Cauchy-Schwartz inequality yields:
 	\begin{align}\label{eqapenexpobound3}
 		\sum_{q=1}^\infty \frac{E\Big[\Big|\int_{t'}^t \sqrt{12}b_{1,n}'(u,x+\sigma\cdot B_u) \diffns u\Big|^{2q}\Big]}{q!}\leq & \Big(\sum_{q=1}^{\infty}\frac{2^qC_{\sigma,k}^q(1+|x|^{q}) (t-t_0)^{\frac{q-1}{2}}
 			}{\sqrt{q!}2^q} \Big)|t-t'|^{1/2}\notag\\
 			 \leq& \Big(\sum_{q=1}^{\infty}\frac{2^qC_{\sigma,k}^{2q}(1+|x|)^{2q} (t-t_0)^{(q-1)}
 			}{q!} \Big)^{1/2} \Big(\sum_{q=1}^{\infty}\frac{ 1}
 		{4^q} \Big)^{1/2}|t-t'|^{1/2}\notag\\
 		 \leq & C \Big(\sum_{q=1}^{\infty}\frac{2^qC_{\sigma,k}^{2q}(1+|x|)^{2q} T^{(q-1)}
 		}{q!} \Big)^{1/2} |t-t'|^{1/2}
 	\notag\\
 	\leq & \frac{C}{\sqrt{T}} \exp\{C_{\sigma,k} T(1+|x|)^2\}|t-t'|^{1/2}	.
 			\end{align}

 	Similarly, it can be proved that $E\Big[e^{16 \int_t^s b'_{1,n}(u, x + \sigma \cdot B_u)\,du} \Big]$ is bounded by $\frac{C}{\sqrt{T}}\exp\left\{C_{\sigma,k} T(1 + |x|)^2 \right\}|t-s|^{1/2}$.
 	Therefore there exists a constant $C$ depending on $\sigma$ such that
 	\begin{equation*}
 		E[I^2_1] \le \frac{C}{\sqrt{T}}\exp\left\{C_{\sigma,k} T(1 + |x|)^2 \right\}|t-t'|^{1/2}.
 	\end{equation*}

 	Repeated application of the H\"older inequality yields
 	\begin{align}\label{mallderI2r}
 		E[I_2^2]&= E\Big[ \Big(\int_{t'}^tD^i_{t'}b_2(u,X_u^{n,x},\omega)e^{-\int_{s}^u b_{2}'(r,X^{x,n}_r,\omega)\diffns r}e^{-\int_{s}^u b_{1,n}'(r,X^{x,n}_r) \diffns r}\diffns u\Big)^2\Big]\notag\\
 		&\leq   E\Big[ \int_{t'}^t\Big(D^i_{t'}b_2(u,X_u^{n,x},\omega)\Big)^2\diffns u \Big(\int_{t'}^te^{-4\int_{s}^u b_{2}'(r,X^{x,n}_r,\omega)\diffns r}e^{-\int_{s}^u4 b_{1,n}'(r,X^{x,n}_r) \diffns r}\diffns u\Big)^{1/2}\Big](t - t')^{1/2}\notag\\
 			&\leq   E\Big[ \Big(\int_{t'}^t\Big(D^i_{t'}b_2(u,X_u^{n,x},\omega)\Big)^2\diffns u\Big)^2\Big]^{\frac{1}{2}}  E\Big[\int_{t'}^te^{-4\int_{s}^u b_{2}'(r,X^{x,n}_r,\omega)\diffns r}e^{-\int_{s}^u4 b_{1,n}'(r,X^{x,n}_r) \diffns r}\diffns u\Big]^{\frac{1}{2}}(t - t')^{1/2} \notag\\
 			&\leq  C |t-t'|^{1/2}E\Big[ \Big(\int_{t'}^t\Big(\tilde{M}_2(u,t',\omega)\Big)^2\diffns u\Big)^2\Big]^{\frac{1}{2}}  E\Big[\int_{0}^Te^{-4\int_{s}^u b_{2}'(r,X^{x,n}_r,\omega)\diffns r}e^{-\int_{s}^u4 b_{1,n}'(r,X^{x,n}_r) \diffns r}\diffns u\Big]^{\frac{1}{2}}\notag\\
 			&\leq  C _{T,b_2^{\mathrm{power}}}|t-t'|^{1/2}  E\Big[\int_{0}^Te^{-4\int_{s}^u b_{2}'(r,X^{x,n}_r,\omega)\diffns r}e^{-\int_{s}^u4 b_{1,n}'(r,X^{x,n}_r) \diffns r}\diffns u\Big]^{\frac{1}{2}}.
 	\end{align}
	Again, using Girsanov transform and H\"older inequality, 
	similar reasoning as before gives 
	$$E[I_2^2]\leq C _{T,b_2^{\mathrm{power}}}\frac{C}{\sqrt{T}} \exp\{C_{\sigma,k} T(1+|x|)^2\}|t-t'|^{\frac{1}{2}}$$ 
	for $0 \leq t' \leq t \leq T$, with $T$ small enough. 
 			 	
 	As for $I_3$, once more repeated use of Cauchy-Schwartz inequality and assumption \ref{a1} give the existence of a constant $C$ that may change from one line to the other such that 
 	\begin{align}\label{eq: Db2 bounded} %
 		E[I_3^2] &=  E\Big[\Big(\int_{t}^s\Big(D^i_{t'}b_2(u,X^{x,n}_u,\omega)-D^i_{t}b_2(u,X^{x,n}_u,\omega)\Big)e^{-\int_{s}^u b_{1,n}'(r,X^{x,n}_r) \diffns r}e^{-\int_{s}^u b_2'(r,X^{x,n}_r,\omega) \diffns r}\diffns u\Big)^2\Big]\notag\\
 		&\le CE\Big[\int_{t}^s|D^i_{t'}b_2(u,X^{x,n}_u,\omega)-D^i_{t}b_2(u,X^{x,n}_u,\omega)|^4\diffns u \Big]^{1/2}E\Big[\int_t^se^{-4\int_{s}^u b_{1,n}'(r,X^{x,n}_r) \diffns r}e^{-4\int_{s}^u b_2'(r,X^{x,n}_r,\omega) \diffns r}\diffns u \Big]^{1/2}\notag\\
 		&\le CE\Big[\int_{t}^se^{-8\int_{t'}^u b_{1,n}'(r,X^{x,n}_r) \diffns r}\diffns u\Big]^{\frac{1}{4}}E\Big[\int_{t}^se^{-8\int_{t'}^u b_2'(r,X^{x,n}_r,\omega) \diffns r}\diffns u\Big]^{1/4}|t-t'|^{\alpha}.
 	\end{align}
 	Once again, using Girsanov theorem and the linear growth condition on the drift $b_1$, one can show that the expectations $E\Big[\int_{t}^se^{-\int_{s}^u8 b_{1,n}'(r,X^{x,n}_r) \diffns r}\diffns u\Big]$ is bounded by $\frac{C}{\sqrt{T}} \exp\{C_{\sigma,k} T(1+|x|)^2\}$. %Using assumption \ref{a1}, we can conclude that there exists $\alpha>0$ such that $E\Big[\int_{t}^s\Big(D^i_{t'}b_2(u,X^{x,n}_u,\omega)-D^i_{t}b_2(u,X^{x,n}_u,\omega)\Big)^4\diffns u\Big]^{\frac{1}{4}}\leq C|t-t'|^{\alpha}$. 
 	Moreover, the assumptions on $b_2$ insure that the two last integral terms on the right hand side of \eqref{eq: Db2 bounded} are bounded by $C$. Therefore, there exists $\alpha >0$ such that 
 	\begin{equation}
 	\label{mallderI3r}
 		E[I_3^2]\leq C \frac{C}{\sqrt{T}} \exp\{C_{\sigma,k} T(1+|x|)^2\}|t-t'|^{\alpha}
 	\end{equation} 
 	for $0 \leq t' \leq t \leq T_1$ with $T_1$ small enough.

 	Combining \eqref{mallderI1r}-\eqref{mallderI3r}, there exists a function ${\cal C}= {\cal C}(k, |x|^2,b_2^{\mathrm{power}})>0$ depending on $k, b_2$ and $ |x|^2$ such that 
 	$$
 		E \left[ | D_t X^{x,n}_s - D_{t'} X^{x,n}_s |^2 \right] \leq {\cal C}(k, |x|^2,b_2^{\mathrm{power}})|t -t'|^{\alpha'}
 	$$
 	for $0 \leq t' \leq t \leq T$ with $T$ small enough ($T\le 1\wedge T_1$) and $\alpha' =\min(\alpha,1/2)$. Thus the first part of the Lemma is shown.
 	Taking $t'>s$ above, $D_{t'}X_s^{x,n}=0$, which implies 
 	$$
 		\sup_{0 \leq t \leq T} E \left[ | D_t X^{x,n}_s |^2 \right] \leq {\cal C}(k, |x|^2,b_2^{\mathrm{power}}).
 	$$
 	This proves the lemma.
\end{proof}

\paragraph{2.2.3. Weak convergence to the weak solution.}
In this step, we show that for each $0 \leq t \leq T$ the above sequence $(X^{x,n}_{t} )_{n \geq 1}$ converges weakly to $E\Big[X^x_t|\mathcal{F}_t\Big]$ in the space $L^2(\Omega,P;\mathcal{F}_t)$. 
 		
\begin{lemma}
\label{lem:weak conv weak sol}
	Assume $b^{\text{exp}}_2<\infty$ and $\Omega$ is the canonical space.
	Choose the sequence $b_{1,n}: [0,T] \times \mathbb{R} \rightarrow \mathbb{R}$, $n\ge 1 $ as before, and let $(X^{x,n}_{t})_{n \geq 1}$ be the corresponding strong solutions to the SDE \eqref{eqmainr}. Then for each $0 \leq t \leq T$, and each function $h:\mathbb{R}\to \mathbb{R}$ of polynomial growth, the sequence $(h(X^{x,n}_{t}))_{n\geq 1}$ is uniformly bounded in $L^2(\Omega,P;\mathcal{F}_t)$ and converges weakly to $E\Big[h(X^x_t)|\mathcal{F}_t\Big]$ in this space.
\end{lemma}
\begin{proof}
	Let us first show that $(h(X^{x,n}_{t}))_{n\geq 1}$ is bounded in 
	$L^2(\Omega,P;\mathcal{F}_t)$. In fact, using Girsanov transform, H\"older inequality and the fact that $(1+|z|^p)e^{-|z|^2 /2s}$ can be bounded by $C_pe^{-\frac{|z|^2}{ 2^{p+1}s}}$, where $C_p$ is a constant depending on $p$, we have
	\begin{align}\label{eqweaklim1}
		\sup_{n}	E\left[|h(X^{x,n}_{t})|^2\right] &\leq  E\left[ e^{2\sum_{i=1}^d\int_0^Tu_{i,n}(r,x+\sigma\cdot B_r,\omega)\diffns B_r^i-2\sum_{i=1}^d\int_0^Tu_{i,n}^2(r,x+\sigma\cdot B_r,\omega)\diffns r}\right]^{\frac{1}{2}}\notag\\
		&\quad \times E\left[e^{2\sum_{i=1}^d\int_0^Tu_{i,n}^2(r,x+\sigma \cdot B_r,\omega)\diffns r}\right]^{\frac{1}{4}}E\left[|h(x+\sigma \cdot B_{t})|^4\right]^{\frac{1}{4}}\notag\\
		&\leq  C E\left[|h(x+\sigma \cdot B_{t})|^4\right]^{\frac{1}{4}}\notag\\
		&= C\Big(\frac{1}{\sqrt{2\pi t\|\sigma \|^2}}\int_{\mathbb{R}}|h(x+z)|^4e^{-|z|^2 /2t\|\sigma \|^2}\diffns z\Big)^{\frac{1}{4}}\notag\\
		&\leq  \frac{C_{\|\sigma \|^2}}{(2\pi t\|\sigma \|^2)^{\frac{1}{8}}}\Big(|x|^{4}\int_{\mathbb{R}}e^{-\frac{|z|^2}{ 2t\|\sigma \|^2}}\diffns z +\int_{\mathbb{R}}e^{-\frac{|z|^2}{ 2^{5}t\|\sigma \|^2}}\diffns z \Big)^{\frac{1}{4}}<\infty.
 	\end{align}
 	The constant $C_{\|\sigma\|^2}$ above depends only on $ \|\sigma \|^2$ and $|x|$.	

 	To show that $(h(X^{x,n}_{t}))_{n\geq 1}$  converges weakly to $E\Big[h(X^x_t)|\mathcal{F}_t\Big]$ in $L^2(\Omega,P;\mathcal{F}_t)$, first notice that the space
 	$$
 		\left\{\mathcal{E}\Big(\int_0^T\dot\varphi_u\diffns B_u\Big): \varphi\in C^1_b([0,T],\mathbb{R}^d)\right\}
 	$$
 	is a dense subspace of $L^2(\Omega,P)$. Here $C^1_b([0,T],\mathbb{R}^d)$ is the space of bounded continuous differentiable functions on $[0,T]$ and with values in $\mathbb{R}^d$ and $\dot\varphi$ is the derivative of $\varphi$. Hence, it is enough to show that $\Big(h(X^{x,n}_{t})\mathcal{E}\Big(\int_0^T\dot\varphi_r\diffns B_r\Big)\Big)_{n}$ converges to $E\Big[h(X^x_t)|\mathcal{F}_t\Big]\mathcal{E}\Big(\int_0^T\dot\varphi_r\diffns B_r\Big)$ in expectation. % 
 	Since $\Omega$ is a Wiener space, we know from the Cameron-Martin theorem, see e.g. \cite{UsZa1}, that for every $h$ measurable,
 	\begin{align}
 	\label{eq:CM}
 		E\left[h(X^x_t)\mathcal{E}\left(\int_0^T\dot\varphi_u\diffns B_u\right)\right]=\int_{\Omega} h(X^x_t(\omega+\varphi))\diffns P(\omega).
 	\end{align}
Let $\varphi \in C^1_b([0,T],\mathbb{R}^d)$.
For every $n$, the process $\tilde X^{x,n}$ given by $\tilde X^{x,n}_t(\omega):= X^{x,n}_t(\omega+ \varphi)$ solves the SDE
\begin{equation}
\label{eq:CM sde}
	d\tilde X^{x,n}_t = (b_{1,n}(t,\tilde X^{x,n}_t) + \tilde b_2(t, \tilde X^{x,n}_t,\omega) + \sigma\dot\varphi_t)\diffns t + \sigma\diffns B_t
\end{equation}
where $\tilde b_2(t,x,\omega):= b_2(t,x,\omega+\varphi)$.
To see this, let $H \in L^2(\Omega, P)$ and apply \eqref{eq:CM} and the fact that $X^{x,n}$ solves the SDE \eqref{eqmainr} to get 
\begin{align*}
	 E[\tilde X^{x,n}_tH] &= E\left[X^{x,n}_tH(\omega - \varphi)\mathcal{E}\left(\int_0^T\dot\varphi(u)\diffns B_u\right)\right]\\
	  &= E\Big[\Big(x + \int_0^tb_1(u, X^{x,n}_u) + b_2(u, X^{x,n}_u,\omega)\diffns u + \sigma B_t \Big)H(\omega-\varphi){\cal E}\Big(\int_0^T\dot\varphi\diffns B \Big) \Big] \\
	      & = E\Big[\Big(x + \int_0^t b_1(u, X^{x,n}_u(\omega+ \varphi))+ b_2(u,  X^{x,n}_u(\omega+\varphi),\omega+\varphi)\diffns u  + \sigma B_t(\omega+\varphi)\Big)H\Big]\\
	      &= E\Big[\Big(x + \int_0^t b_1(u, \tilde X^{x,n}_u(\omega))+ \tilde b_2(u,\tilde X^{x,n}_u(\omega),\omega)+ \sigma\dot\varphi\diffns u  + \sigma B_t(\omega)\Big)H\Big],
\end{align*} 
where the last equality follows by the fact that $B_t(\omega + \varphi) = B_t(\omega) + \varphi$, since $B$ is the canonical process.
This proves the claim.
Since $X^x$ satisfies the SDE (without been adapted to the filtration $({\cal F}_t)$), with respect to a probability measure $Q$ which is equivalent to $P$, see the proof of Lemma \ref{lem:weak existence}, the above arguments show that $\tilde X^x(\omega):= X^x(\omega+ \varphi)$ satisfies
\begin{equation}
	d\tilde X^{x}_t = (b_{1}(t,\tilde X^{x}_t) + \tilde b_2(t, \tilde X^{x,n}_u,\omega) + \sigma\dot\varphi_t)\diffns t + \sigma\diffns B_t \quad P\text{-a.s.}
\end{equation}
Now, put
\begin{equation}\label{eqtildeu1}
	\tilde u_{i,n}: =\frac{\sigma_i}{\sigma_1^2+\cdots + \sigma_d^2 }(b_{1,n} + \tilde b_2) =: b_{1,n}^{\sigma_i} + \tilde b_2^{\sigma_i} \quad \text{and} \quad \tilde u_{i}=\frac{\sigma_i}{\sigma_1^2+\cdots + \sigma_d^2 }(b_{1} + \tilde b_2) =; b_1^{\sigma_i} + b_2^{\sigma_i}.
\end{equation}
 	It follows by Girsanov theorem that
 	\begin{align}\label{eqweaklim2}
 		&	E\Big[h(X^{x,n}_{t})\mathcal{E}\Big(\int_0^T\dot\varphi_r\diffns B_r\Big)-E\Big[h(X^x_t)|\mathcal{F}_t\Big]\mathcal{E}\Big(\int_0^T\dot\varphi_r\diffns B_r\Big)\Big]\notag
 		=  E\Big[\Big(h(X^{x,n}_{t})-h(X^x_t)\Big)\mathcal{E}\Big(\int_0^T\dot\varphi_r\diffns B_r\Big)\Big]\notag\\
 		&=  E\Big[h(x+\sigma\cdot B_{t})\Big(\mathcal{E}\Big(\int_0^T\Big\{\tilde u_n(r,x+\sigma \cdot B_r,\omega)+\dot\varphi_r\Big\}\diffns B_r\Big)-\mathcal{E}\Big(\int_0^T\Big\{\tilde u(r,x+\sigma \cdot B_r,\omega)+\dot\varphi_r\Big\}\diffns B_r\Big)\Big)\Big].
 	\end{align}
 	Using the fact that $|e^a -e^b|\leq |e^a + e^b||a - b|$, the H\"older inequality and Burkholder-Davis-Gundy inequality, we get
 	\begin{align}\label{eqweaklim3}
 		&	E\Big[h(X^{x,n}_{t})\mathcal{E}\Big(\int_0^T\dot\varphi_r\diffns B_r\Big)-E\Big[h(X^x_t)|\mathcal{F}_t\Big]\mathcal{E}\Big(\int_0^T\dot\varphi_r\diffns B_r\Big)\Big]\notag\\
 		  &\leq C E\Big[h(x+\sigma\cdot B_{t})^2\Big]^{\frac{1}{2}}E\Big[\Big(\mathcal{E}\Big(\int_0^T\Big\{\tilde u_n(r,x+\sigma \cdot B_r,\omega)+\dot\varphi_r\Big\}\diffns B_r\Big)+\mathcal{E}\Big(\int_0^T\Big\{\tilde u(r,x+\sigma \cdot B_r,\omega)+\dot\varphi_r\Big\}\diffns B_r\Big)\Big)^4\Big]^{\frac{1}{4}}\notag\\
 		&\quad \times\Big\{E\Big[ \Big(\int_0^T\Big(\tilde u_n(r,x+\sigma \cdot B_r,\omega)- \tilde u(r,x+\sigma \cdot B_r,\omega)\Big)\diffns B_r\Big)^4\Big]\notag\\
 		&\quad +E\Big[\Big(\int_0^T\|\tilde u_n(r,x+\sigma \cdot B_r,\omega)+\dot\varphi(t)\|^2	-\|\tilde u(r,x+\sigma \cdot B_r,\omega)+\dot\varphi_r\|^2\diffns r\Big)^4\Big] \Big\}^{\frac{1}{4}}\notag\\
 		&=I_1\times I_{2,n}\times (I_{3,n}+I_{4,n}).
 	\end{align}
 	That $I_1$ is finite was proved in the computations leading to  \eqref{eqweaklim1}.
 	Observe that 
 	\begin{align*}
 		&\mathcal{E}\Big(\int_0^T\left\{\tilde u_n(r,x+\sigma \cdot B_r,\omega)+\dot\varphi_r\right\}\diffns B_r\Big) = {\cal E}\Big(\int_0^T b_{1,n}^\sigma(r, x +\sigma\cdot B_r,\omega)\diffns B \Big){\cal E}\Big(\int_0^T \tilde b_2^\sigma\diffns B \Big){\cal E}\Big(\int_0^T\dot\varphi\diffns B \Big)\\
 		&\quad\times\exp\Big(-\int_0^T\dot\varphi_rb_{1,n}^\sigma(r, x +\sigma\cdot B_r,\omega) - \dot\varphi_r\tilde b_2^\sigma - \tilde b^\sigma_2b_{1,n}^\sigma(r, x +\sigma\cdot B_r,\omega)\diffns r \Big).
 	\end{align*}
 	 Thus, $ I_{2,n}$ is bounded by similar argument as in the proof of  Lemma \ref{lemmainres1r} since $\dot\varphi$ is bounded.
 	Using the dominated convergence theorem, we get that $I_{3,n}$ and $I_{4,n}$ converge to $0$ as $n$ goes to infinity.
\end{proof}
The following result is a corollary of the compactness result given by Lemma \ref{lemmainres1r} and Corollary \ref{maincor}.

\begin{proposition}
\label{prop:convXn}
	For any fixed $ t \in [0,T]$, with $T$ small and $x \in \mathbb{R}$, the sequence $(X^{n,x}_t)_{n\geq1}$ of strong solutions to the SDE \eqref{eqmainr} converges strongly in $L^2(\Omega,P;\mathbb{R})$ to $E\Big[X^x_t|\mathcal{F}_t\Big]$.
\end{proposition}
 			
\begin{proof}
	Observe that by the compactness criteria for each $t$, there exists a subsequence $(X^{x,n_k}_t)_{k\geq1}$ that converges strongly in $L^2(\Omega,P)$.  From the previous lemma, we get by setting $h(x)=x, x\in \mathbb{R}$ that $(X^{x,n_k}_t)_{n\geq1}$ converges weakly to $E\Big[X^x_t|\mathcal{F}_t\Big]$ in $L^2(\Omega,P)$ and therefore by the uniqueness of the limit there exists a subsequence $n_k$ such that $(X^{x,n_k}_t)_{n\geq1}$ converges strongly to $E\Big[X^x_t|\mathcal{F}_t\Big]$ in $L^2(\Omega,P)$. The convergence then holds for the entire sequence by uniqueness of limit. Indeed, by contradiction, suppose that for some $t$, there exist $\epsilon >0$ and a subsequence $n_l, l\geq 0$ such that  
	$$
		\|X_t^{x,n_l}-E[X^x_t|\mathcal{F}_t]\|_{L^2(\Omega,P)}\geq \epsilon.
	$$
	We also know by the compactness criteria that there exists a further subsequence of $n_m, m\ge 0$ of $n_l, l\geq 0$ such that 
	$$
		X_t^{x,n_{n_m}} \text{ converges to } \tilde{X}_t \text{ in } L^2(\Omega,P) \text{ as } m \text{ goes to } \infty .
	$$
	Nevertheless, $(X^{x,n_k}_t)_{n\geq1}$ converges weakly to $E\Big[X^x_t|\mathcal{F}_t\Big]$ in $L^2(\Omega,P)$, and hence by the uniqueness of the limit, we have 
	$$
		\tilde{X}_t=E\Big[X^x_t|\mathcal{F}_t\Big].
 	$$ 
	Since 
 	$$
 		\|X_t^{x,n_{n_m}}-E[X^x_t|\mathcal{F}_t]\|_{L^2(\Omega,P)}\geq \epsilon,
 	$$
 	this is a contradiction.
\end{proof}

\paragraph{2.2.4. Adaptedness of the weak solution and uniqueness.} 			
Finally, we show that the weak solution $X^x_t $ is $(\mathcal{F}_t)_{ t\in [0,T]}$-adapted and unique.
 			
\begin{theorem}
\label{thm:adapted}
	The weak solution $X^x_t $ to the SDE \eqref{eqmain1r} is $\mathcal{F}_t$-measurable. 
\end{theorem}
\begin{proof}
	Let us first show that $X^x_t$ is $\mathcal{F}_t$-measurable. 
	Let $h$ be a globally Lipschitz continuous function. 
	By Proposition \ref{prop:convXn}, there exists a subsequence $n_k,k\geq 0$, such that $h(X^{x,n_k}_t)$ converges to $h(E[X^x_t|\mathcal{F}_t]) \,P$-a.s. as $k$ goes to infinity. Moreover, we know that $h(X^{x,n_k}_t)$ converges to $E[h(X^x_t)|\mathcal{F}_t]$ weakly in $L^2(\Omega,P)$ as $k$ goes to infinity. We get from the uniqueness of the limit that 
	$$
 		h(E[X^x_t|\mathcal{F}_t])=E[h(X^x_t)|\mathcal{F}_t]\,\,\,P\text{-a.s.}
 	$$
 	Since the above holds for any arbitrary globally Lipschitz continuous function, it follows that $X^x_t$ is $\mathcal{F}_t$-measurable.
\end{proof}
\begin{proposition}
\label{pro:uniqueness}
	The SDE \eqref{eqmain1r} satisfies the pathwise uniqueness property.
\end{proposition}
\begin{proof}
	Let $X^x_t$ and $\tilde X^x_t$ be two solutions to the SDE \eqref{eqmain1r}. 	For $\varphi \in C^1_b([0,T],\mathbb{R}^d)$, 
	we have  
 	\begin{align}
 		E\Big[X^x_t\mathcal{E}\Big(\int_0^T\dot\varphi_u\diffns B_u\Big)\Big]=\int_{\Omega} X^x_t(\omega+\varphi)\diffns P(\omega),
 	\end{align}
 	where, as shown in the course of the proof of Lemma \ref{lem:weak conv weak sol}, $X^x_t(\omega+\varphi)$ satisfies the SDE \eqref{eq:CM sde} .
 	Similarly $\tilde X^x_t(\omega+\varphi)$ satisfies the same SDE. Thus, since the drift is of linear growth, it follows that $(X^x_t(\omega+\varphi),B)$ and $(\tilde X^x_t(\omega+\varphi),B)$ have the same distribution.
 	In fact, using that the distributions $P^x$ and $\tilde P^x$ of $X^x$ and $\tilde X^x$, respectively are equal to $P$ (see the construction in \cite[Proposition 3.6]{KaShr88}) it follows that by assumptions on $b_2$ and the linear growth of $b_1$ that $\int_0^T|b_1(t,X_t^x) + (\sigma\dot\varphi_t +b_2(t,X_t^x,\omega+\varphi))|^2\diffns t<\infty$ $P^x$-a.s.
 	The same holds if $X^x$ is replaced by $\tilde X^x$ and $P^x$ by $\tilde P^x$.
 	Thus, a simple adaptation of the proof of \cite[Porposition 3.10]{KaShr88} shows that $(X^x_t(\omega+\varphi),B)$ and $(\tilde X^x_t(\omega+\varphi),B)$ have the same distribution.
 	Hence, for all $t,\varphi$, we have $E\Big[\tilde X^x_t\mathcal{E}\Big(\int_0^T\dot\varphi_u\diffns B_u\Big)\Big]=E\Big[X^x_t\mathcal{E}\Big(\int_0^T\dot\varphi_u\diffns B_u\Big)\Big]$, from which pathwise uniqueness follows.
\end{proof}

\paragraph{2.2.5. Global existence.}
Since the small time $T_1$ for which the solution exists does not depend on the initial condition (see \eqref{eq:cond small time 1}) one can use a standard pasting argument to show that the solution exists for all time $T>0$.
In addition
using the linear growth condition on $b_1$ and the integrability condition on $b_2$, it follows from Gronwall lemma that the unique solution does not explode. 

The proof of Theorem \ref{thmainres1r} is now complete.
\hfill$\Box$
\begin{remark}
\label{rem:muldi-d}
	For the solvability of the SDE \eqref{eq:sde into} in the multi-dimensional case, the main  technical hurdle is the use of the explicit solution of the (random) ODE \eqref{malldifeqr} in the proof of Lemma \ref{lemmainres1r}, which does not carry-over to the multi-dimensions.
	An alternative approach could be to employ the iteration procedure used in \cite{MeMo17}.
	But in that case, due to the randomness of the drift, the estimations do not reduce to the applications of the key integration by parts estimate in Proposition \ref{pro:uniqueness} which in turn was an important argument in the proof of compactness of the approximating sequence.
\end{remark}
\section{Malliavin differentiability}
\label{sec:SDE.Malliavin}
\subsection{Differentiability of the strong solution}

In this subsection, we show that the unique strong solution of the SDE \eqref{eqmain1r} constructed in the previous subsection is Malliavin differentiable and we derive a representation formula of the Malliavin derivative.
The proof Malliavin differentiability for small time interval follows directly from Lemma \ref{lemmainres1r}. 
In order to control the Malliavin derivative of the process on arbitrary time intervals, we need the following result which is a variant of Fernique theorem whose proof is similar to that of \cite[Lemma 2.6]{MeMo17}.

\begin{lemma} \label{lemmaproexpr1}
	Let $t_0 \in [0, 1]$ and let $\eta : \Omega \to \mathbb{R}$ be a $\mathcal F_{t_0}$-measurable random variable independent of the $P$-augmented filtration generated by the Brownian motion $B$. 
	Let $b: [0,1] \times \Omega \times \mathbb{R}\rightarrow \mathbb{R}$ such that $b=b_1+b_2$ satisfies \ref{a1}, with $b_1$ a smooth coefficient with compact support satisfying a global linear growth condition. 
 	Denote by $X^{t_0,\eta}_{t}$ the unique strong solution \textup{(}if it exists\textup{)} to the SDE \eqref{eqmain1r} starting at $\eta$ and with drift coefficient $b$. Then we can find a positive  number $\delta_0$ independent of $t_0$ and $\eta$ \textup{(}but may depend on $||\tilde b_1||_\infty$\textup{)} such that
 	\begin{align}\label{eqsupproAr1}
 		E\exp \{\delta_0 \sup_{t_0\leq t\leq 1}|X^{n,t_0,\eta}_{t}|^2 \} \leq C_1 E\exp \{C_2\delta_0|\eta|^2 \},
 	\end{align}
 	where $C_1, C_2$ are positive constants independent of $\eta$, but may depend on $k_1$ and $b_2$. Moreover, $C_1$ may depend on $\delta_0$. Furthermore, if the right hand side of \eqref{eqsupproAr1} is finite then the above expectation is finite.
\end{lemma}

\begin{proof}
	We have the following almost sure equality
 	\begin{equation}
 		\begin{array}{ll}\label{eqlemest1}
 			X^{t_0,\eta}_{t} = \eta +  \displaystyle \int_{t_0}^{t} \Big(b _1(u,X^{t_0,\eta}_{u})+ b_2(u,X^{t_0,\eta}_{u},\omega)  \Big)\diffns u + \sigma\cdot (B_t -B_{t_0}),  \quad t_0 \leq t \leq 1.
 		\end{array}
 	\end{equation}
 	Successive application of H\"older's inequality to \eqref{eqlemest1} yields
 	\begin{align} \label{eqlemest2}
 		|X^{t_0,\eta}_{t}|^2 &\leq 4|\eta|^2 + 4 \Big | \int_{t_0}^{t} b (u,X^{t_0,\eta}_{u}) \diffns u \Big |^2 +4 \Big | \int_{t_0}^{t} b_2(u,X^{t_0,\eta}_{u},\omega) \diffns u \Big |^2 + 4|\sigma|^2|B_t -B_{t_0}|^2\notag  \\
 		&\leq  4|\eta|^2 + 4 \Big(\int_{t_0}^{t} k_1(1+ |X^{t_0,\eta}_{u}|) \diffns u  \Big )^2 +4 \Big | \int_{t_0}^{t} b_2(u,X^{t_0,\eta}_{u},\omega) \diffns u \Big |^2 + 4|\sigma|^2|B_t -B_{t_0}|^2\notag \\
 		&\leq   4|\eta|^2 + 8 k^2_1  (t-t_0)\int_{t_0}^{t} \{1+ |X^{t_0,\eta}_{u}|^2 \} \diffns u + 4 (t-t_0)\int_{t_0}^{t} |b_2(u,X^{t_0,\eta}_{u},\omega)|^2 \diffns u + 4|\sigma|^2|B_t -B_{t_0}|^2 \notag\\
		&\leq 4|\eta|^2 + 8 k^2_1 (t-t_0)^2 + 8 k^2_1 (t-t_0) \int_{t_0}^{t} |X^{t_0,\eta}_{u}|^2 \diffns u 
 		 +4 (t-t_0)^2 C^2_T|M_2(\omega)|^2 +4|\sigma|^2|B_t -B_{t_0}|^2,   \text{ a.s. }
 	\end{align}
 	Take the supremum on both sides of \eqref{eqlemest2} and multiply by $\delta_0$ to get
  	\begin{align}\label{eqlemest3}
 		\delta_0\sup_{t_0\leq t\leq 1} |X^{t_0,\eta}_{t}|^2
 		\leq	&  4\delta_0|\eta|^2 + 8 k^2_1 \delta_0+ 8 k^2_1  \int_{t_0}^{1}\delta_0 \sup_{0\leq u\leq s} |X^{t_0,\eta}_{u}|^2 \diffns s \notag\\
 		&+4 \delta_0 (t-t_0)^2 C^2_T|M_2(\omega)|^2  +4 |\sigma|^2\delta_0\sup_{t_0 \leq t \leq 1}|B_t -B_{t_0}|^2 ,   \text{ a.s. }
 	\end{align}
 	Applying Gronwall's lemma to \eqref{eqlemest3}, we have
 	\begin{align}\label{eqlemest4}
 		\delta_0\sup_{t_0\leq t\leq 1} |X^{t_0,\eta}_{t}|^2
 		&\leq \left\{4\delta_0|\eta|^2 + 8 k^2_1 \delta_0 +4 \delta_0|\sigma|^2\sup_{t_0 \leq u \leq 1}|B_t -B_{t_0}|^2  
 		 +4 \delta_0  |M_2(\omega)|^2	 \right\}  e^{8 k^2_1 }  ,   \text{ a.s. }
 	\end{align}
 	Now, set $C_2:=4e^{8 k^2_1}$. Then taking exponential on both sides of \eqref{eqlemest4}, we have 
 	\begin{align}\label{eqlemest5}
 		\exp \left\{\delta_0\sup_{t_0\leq t\leq 1} |X^{t_0,\eta}_{t}|^2 \right\}
 		&\leq  \exp \left\{2C_2\delta_0k^2\right\} \times \exp \{\delta_0 C_2|\eta|^2\}\times \exp \left\{C_2\delta_0|\sigma|^2\displaystyle\sup_{t_0 \leq u \leq 1}|B_u -B_{t_0}|^2 \right\} \notag\\
 		&\quad \times \exp \left\{C_2\delta_0|M_2(\omega)|^2	\right\},\quad P   \text{-a.s.}
 	\end{align}
 	Taking expectations in the above  and using once more H\"older inequality, we get
 	\begin{align}\label{eqlemest6}
 		E\exp \left\{\delta_0\sup_{t_0\leq  t\leq 1}|X^{t_0,\eta}_{t}|^2\right\}
 		&\leq \exp \left\{2C_2k^2\delta_0\right\} \cdot E\left[\exp \left\{3C_2\delta_0|\eta|^2\right\}\right]^{\frac{1}{3}}\times E\left[\exp \left\{3C_2\delta_0\|\sigma\|^2\sup_{t_0 \leq u \leq t}|B_u -B_{t_0}|^2 \right\} \right]^{\frac{1}{3}}\notag\\
 		&\quad \times E\left[ \exp \left\{3C_2\delta_0|M_2(\omega)|^2\right\}\right]^{\frac{1}{3}}  .
 	\end{align}
 	The result follows provided that we find  $\delta_0$ independent of $\eta$ and $t_0$ such that
 	\begin{equation}\label{eqlemest7}
 		\begin{array}{ll}
 			E\left[\exp \left\{3C_2\delta_0|\sigma|^2\displaystyle\sup_{t_0 \leq u \leq 1 }|B_u -B_{t_0}|^2 \right\}\right]  < \infty   .
 		\end{array}
 	\end{equation}
 	The estimate \eqref{eqlemest7} is obtained from the Fernique theorem.
 	In fact, one can show that for $\delta_0<\min(\frac{1}{12dC_2|\sigma|^2},\frac{1}{C_2})$, the result holds (see for example \cite[Lemma 2.6]{MeMo17} for details).
 	From this \eqref{eqlemest6} yields
 	\begin{align}
 		E\exp \left\{\delta_0\sup_{t_0\leq t\leq 1}|X^{n,t_0,\eta}_{t}|^2 \right\}  \leq C_1 Ee^{C_2\delta_0|\eta|^2}.
 	\end{align}
 	Note that $C_1, C_2$ and $\delta_0$ are independent of $\eta$ and $t_0$ (but may depend on $||\tilde b_1||_\infty$  and $ |\sigma|^2$). Thus \eqref{eqsupproAr1} is valid for the above choice of $\delta_0$.
\end{proof}

Next, using Lemma \ref{lemmaproexpr1}, we prove that under the conditions of Theorem \ref{thmainres1r} the Malliavin derivative is bounded in the $L^2(\Omega,P)$ norm.

\begin{proof}[of Theorem \ref{thm:Mall.diff.SDE}]  
	First recall that by the second part of Lemma \ref{lemmainres1r}, the sequence of strong solutions of the SDE \eqref{eqmainr} 
	satisfies
	$$
	\sup_{0 \leq t \leq T} \sup_{n\geq 1}E \left[ | D_t X^{x,n}_s |^2 \right] \leq{\cal C}(\|\tilde{b}_1\|_{\infty},|x|^2,b_2^{\text{power}}).
	$$
	Since the Malliavin derivative is a closable operator (see e.g. \cite[Exercise 1.2.3]{Nua06}), it follows from Proposition \ref{prop:convXn} and Theorem \ref{thm:adapted} that $X^x_t$ is Malliavin smooth.

	It remains to prove integrability of the derivative.
	This is done by induction. Choose $\delta_0$ as in Lemma \ref{lemmaproexpr1} and define $$ \tau:=\frac{\delta_0}{64d\sqrt{2} k^2}\le T_1,$$ let $s_i = i\tau$ and $x_i:= X^{x,n,0}_{s_i} , 
 	i \geq 1$. It follows form the previous argument that the result is valid for $0\leq t\leq s_1$.

	Assume that there exists a Malliavin differentiable solution $\{X_t,0\leq t\leq s_{m}\}$. Set $t$ such that $s_{m}\leq t< s_{m+1}$.
 	Let	 $(X^{x,n}_{t})_{n \geq 1}$ be the approximating sequence defined by \eqref{eqmainr} it satisfies the a.s. relation
 	\begin{equation*}
 		\begin{array}{ll}
 		X^{x_m,n}_{t} = X_{s_m}^n +  \displaystyle \int_{s_m}^{t} b_n (u,X^{x_m,n}_{u},\omega)  \diffns u + \sigma(B_t -B_{s_m}),  \quad  s_m\leq t \leq s_{m+1}.
 		\end{array}
 	\end{equation*}
 			
 	Using the chain-rule for the Malliavin derivatives, we have component wise
 	\begin{equation}
 		D^i_sX^{x_m,n}_t =\left\{\begin{array}{llll}
 			D^i_sX_{s_m}^{n}  + \int_{s_m}^t \Big(\{b_{1,n}'(u,X^{x_m,n}_u)+b_{2}'(u,X^{x_m,n}_u,\omega)\}D^i_s X^{x_m,n}_u  +D^i_sb_2(t,X^{x_m,n}_u,\omega)\Big)\diffns u  & \text{ if } s\leq s_m\\
 			\sigma_i +	\int_{s}^t \Big(\{b_{1,n}'(u,X_u)+b_{2}'(u,X^{x_m,n}_u,\omega)\}D^i_s X^{x_m,n}_u +D^i_sb_2(t,X^{x_m,n}_u,\omega) \Big)\diffns u  &  \text{ if } s>s_m
 			\end{array}\right.,
 	\end{equation}
 	 $P\text{-a.s.}$
 	for all $0 \leq s \leq t$ and solving explicitly gives
%\begingroup\makeatletter\def\f@size{7.5}\check@mathfonts
%\def\maketag@@@#1{\hbox{\m@th\large\normalfont#1}}%
	\begin{align} 
 		&D^i_sX^{x_m,n}_t	=\notag\\
 		&\begin{cases}
 		e^{\int_{s_m}^t \{b_{1,n}'(u,X^{x_m,n}_u) +b_{2}'(u,X^{x_m,n}_u)\}\diffns u}\Big(\int_{s_m}^tD^i_sb_2(u,X^{x_m,n}_u)e^{-\int_{s_m}^u \{b_{1,n}'(r,X^{x_m,n}_r) +b_{2}'(r,X^{x_m,n}_r)\}\diffns r}\diffns u +  D^i_sX_{s_m}^n \Big) & \text{ if } s\leq s_m\\
 			e^{\int_t^s \{b_{1,n}'(u,X^{x_m,n}_u) +b_{2}'(u,X^{x_m,n}_u)\}\diffns u}\Big(\int_t^sD^i_sb_2(u,X^{x_m,n}_u)e^{-\int_t^u \{b_{1,n}'(r,X^{x_m,n}_r)+b_{2}'(r,X^{x_m,n}_r)\} \diffns r}\diffns u + \sigma_i\Big) &  \text{ if } s>s_m
 		\end{cases}
 		\label{eqthemainr1}
 	\end{align}
%	\endgroup
$P\text{-a.s.}$	for all $0 \leq s \leq t$.
	Let us first focus on $D_sX^{x_m,n}_t,$ when $s\leq s_m$. Using H\"older inequality, we have 
 	\begin{align*}
 			E\left[|D^i_sX^n_t|^2\right]	&\leq  2 E\Big[e^{\int_{s_m}^t 2\{b_{1,n}'(u,X^{x_m,n}_u)+b_{2}'(u,X^{x_m,n}_u,\omega)\} \diffns u}\Big(\int_{s_m}^tD^i_sb_2(u,X^{x_m,n}_u,\omega)e^{-\int_{s_m}^u\{ b_{1,n}'(r,X^{x_m,n}_r)+b_{2}'(r,X^{x_m,n}_r,\omega)\}  \diffns r}\diffns u\Big)^2\Big] \notag\\
 				&\quad + 2E\Big[ e^{2\int_{s_m}^t \{b_{1,n}'(u,X^{x_m,n}_u)+b_{2}'(r,X^{x_m,n}_r,\omega)\}  \diffns u}|D^i_sX_{s_m}^n |^2\Big]\notag\\
 				&\leq  2 E\Big[e^{\int_{s_m}^t 2\{b_{1,n}'(u,X^{x_m,n}_u)+b_{2}'(u,X^{x_m,n}_u,\omega)\} \diffns u}\Big(\int_{s_m}^tD^i_sb_2(u,X^{x_m,n}_u,\omega)e^{-\int_{s_m}^u\{ b_{1,n}'(r,X^{x_m,n}_r)+b_{2}'(r,X^{x_m,n}_r,\omega)\}  \diffns r}\diffns u\Big)^2\Big] \notag\\
 				&\quad + 4E\Big[  e^{\int_{s_m}^t 4b_{1,n}'(u,X^{x_m,n}_u) \diffns u}|D^i_sX_{s_m}^n |^2\Big] + E\Big[ e^{4\int_{s_m}^t b'_{2}(r,X^{x_m,n}_r,\omega)  \diffns u}|D^i_sX_{s_m}^n |^2\Big]\notag\\
 			&= 2 E\Big[e^{\int_{s_m}^t 2\{b_{1,n}'(u,X^{x_m,n}_u)+b_{2}'(u,X^{x_m,n}_u,\omega)\} \diffns u}\Big(\int_{s_m}^tD^i_sb_2(u,X^{x_m,n}_u,\omega)e^{-\int_{s_m}^u\{ b_{1,n}'(r,X^{x_m,n}_r)+b_{2}(r,X^{x_m,n}_r,\omega)\}  \diffns r}\diffns u\Big)^2\Big] \notag\\
&\quad +4 		E\Big[ 	E\Big[ e^{\int_{s_m}^t 4b_{1,n}'(u,X^{x_m,n}_u) \diffns u}|\mathcal{F}_{s_m}\Big]|D^i_sX_{s_m}^n |^2\Big] +4E\Big[ e^{8\int_{s_m}^t b_{2}'(r,X^{x_m,n}_r,\omega)  \diffns u}\Big]^{\frac{1}{2}}E\Big[ |D^i_sX_{s_m}^n |^4\Big]^{\frac{1}{2}}\notag\\
 			=&	I_1+I_2+ I_3.
 	\end{align*}
 	Let us now consider the conditional expectation part in $I_2$. As before, using Girsanov theorem and H\"older inequality, we have
 	\begin{align}\label{approxiexpr1}
 	& E\Big[ e^{\int_{s_m}^t 4b_{1,n}'(u,X^{x_m,n}_u) \diffns u}|\mathcal{F}_{s_m}\Big] 	\notag\\
 		&\leq E\Big[ \mathcal{E}\Big(\int_{s_m}^tu_n(u,\omega,X_{s_m}^n+\sigma\cdot(B_u-B_{s_m}))\diffns B_u\Big)e^{\int_{s_m}^t 4b_{1,n}'(u,X_{s_m}^n+\sigma\cdot (B_u-B_{s_m}))\diffns u}|\mathcal{F}_{s_m}\Big]\notag\\
 		&\leq E\Big[ e^{2\sum_{i=1}^d\int_{s_m}^{s_{m+1}}u_{i,n}(r,\omega,X_{s_m}^n+\sigma\cdot (B_r-B_{s_m}))\diffns B_r-2\sum_{i=1}^d\int_{s_m}^{s_{m+1}}u_{i,n}^2(r,\omega,X_{s_m}^n+\sigma\cdot(B_r-B_{s_m}))\diffns r}|\mathcal{F}_{s_m}\Big]^{\frac{1}{2}} \notag\\
 		&\quad \times E\Big[e^{6\sum_{i=1}^d\int_{s_m}^{s_{m+1}}u_{i,n}^2(r,\omega,X_{s_m}^n+\sigma\cdot(B_r-B_{s_m}))\diffns r}|\mathcal{F}_{s_m}\Big]^{\frac{1}{4}} E\Big[e^{16\int_{s_m}^t b_{1,n}'(u,X_{s_m}^n+\sigma\cdot (B_u-B_{s_m})) \diffns u}|\mathcal{F}_{s_m}\Big]^{\frac{1}{4}}.
 	\end{align}
 	By Girsanov theorem applied to the martingale $2\int_{s_m}^{s_{m+1}}b_n(r,X_{s_m}^n+B_r-B_{s_m},\omega)\diffns B_r$ the first term of the right hand side is equal to one. Since $b_{1,n}$ is of spatial linear growth, $B_r-B_{s_m}$ is independent of $\mathcal{F}_{s_m}$ and $X_{s_m}^n$ is $\mathcal{F}_{s_m}$-measurable, it follows from the H\"older inequality, the exponential expansion and the choice of $\tau$ that there exists a constant $C>0$ such that  
 	\begin{align} \label{approxiexpr2}
 		E\left[e^{16\int_{s_m}^{s_{m+1}}u_{i,n}^2(r,X_{s_m}^n+\sigma\cdot (B_u-B_{s_m}),\omega) \diffns r}|\mathcal{F}_{s_m}\right]^{\frac{1}{4}}\leq Ce^{6k^2_1\tau(1+|X_{s_m}^n|^2)}.
 	\end{align}
 	Next let us consider the last term. We can show as in \cite[Proposition 4.10 and (4.28)]{MeMo17} that there exists a constant $C>0$ such that
 	\begin{align} \label{approxiexpr3}
 		E\left[e^{16\int_{s_m}^t b_{1,n}'(u,X_{s_m}^n+\sigma\cdot (B_u-B_{s_m}))  \diffns u}|\mathcal{F}_{s_m}\right]^{\frac{1}{4}}\leq e^ {\tau C k_1|X_{s_m}^n| }.
 	\end{align}
 	Combining \eqref{approxiexpr1}-\eqref{approxiexpr3} and using H\"older inequality, we get
 	\begin{align} \label{approxiexpr4}
 		I_2&\leq CE\left[e^{6k^2_1\tau(1+|X_{s_m}^n|^2)}e^ {\tau C k_1|X_{s_m}^n| }|D^i_sX_{s_m}^n |^2 \right]\notag\\
  			&\leq CE\left[e^{24k^2_1\tau(1+|X_{s_m}^n|^2)}]^{\frac{1}{4}}E[e^ {4\tau C k_1|X_{s_m}^n| }]^{\frac{1}{4}}E[|D^i_sX_{s_m}^n |^4\right]^{\frac{1}{2}}.
 	\end{align}
 	Let us notice that one can show as in the proof of Lemma \ref{lemmainres1r} that there exists $C>0$ such that $E[|D^i_sX_{s_1}^n |^p]\leq C$ for $p\geq 1$ and by induction hypothesis it follows that $E[|D^i_sX_{s_m}^n |^p]\leq C$ for $p=4$. The choice of $\tau$ combined with Lemma \ref{lemmaproexpr1} ensures that one can find $C>0$ such that  $E[e^{24k^2_1\tau(1+|X_{s_m}^n|^2)}]^{\frac{1}{4}}\leq Ce^{C_2\delta_0|x|^2}$. Furthermore, using successive approximation, one can show that $E[e^ {4\tau C k_1|X_{s_m}^n| }]\leq C_1 e^ {C_2 k_1|x| }$, where $C_1$ depends on $b_2$ and $C_2$ is a positive constant. This can be shown as in the proof of Lemma \ref{lemmaproexpr1} using the Gronwall lemma and the probability distribution of the Brownian motion.  

 	The case $s>s_m$ is similar (and easier) since the term $D^i_sX^m_{s^m}$ is not involved, it is replaced by the constant $\sigma$.
 			
 	Since $b_2^{\text{power}}<\infty$, using the H\"older inequality, the term $I_1$ can also be bounded using similar arguments as above. 
 	Thus the Malliavin derivative of the approximating sequence $(X^{x,n}_{t})_{n \geq 1}$  has a uniform bound on $[0,T]$ which does not depend on $n$. 
 	Therefore, $D_sX_t \in L^2(\Omega,P;\mathbb{R}^d)$ for $0\leq s\leq t\leq T$.  
\end{proof}

\begin{remark} Let us observe that if $b_1$ is globally bounded then it follows from the condition on $b_2$ and Girsanov theorem that Lemma \ref{lemmainres1r} holds for all $T$. Therefore, we do not need the above argument and the Malliavin differentiablility of the solution directly follows from the compactness argument. 
 \end{remark}
The following result gives estimates on the Malliavin derivative of the solution.
\begin{theorem} \label{thm:estim diff Mall X}
	Assume that the conditions of Theorem \ref{thmainres1r} are satisfied.
	Then, the Malliavin derivative of the unique strong solution to the SDE \eqref{eqmain1r} satisfies:
	$$
		E \left[ | D_t X^{x}_s - D_{t'} X^{x}_s |^2 \right] \leq {\cal C}(\|\tilde{b}_1\|_{\infty},|x|^2,b_2^{\text{power}}) |t -t'|^{\alpha}
	$$
	for $0 \leq t' \leq t \leq T_1$, $\alpha = \alpha(s) > 0$ and
	$$
 		\sup_{0 \leq t \leq T_1} E \left[ | D_t X^{x}_s |^2 \right] \leq {\cal C}(\|\tilde{b}_1\|_{\infty},|x|^2, b^{\text{power}}_2),
	$$
	where ${\cal C}: [0, \infty)^3 \rightarrow [0, \infty)$ is continuous and increasing in each component and can differ from line to line, $T_1$ is small, and $||\tilde b_1||_\infty$ defined in \eqref{eq:def b infty}.
\end{theorem}
					
\begin{proof} 

	Let $b_{1,n}$ be a sequence of smooth functions with compact support converging a.e. to $b_1$ and such that $b_{1,n}$ satisfies a uniform global linear growth; that is $\|\tilde{b}_{1,n}\|_\infty\leq \|\tilde{b}_{1}\|_\infty$, see \eqref{eq:def b infty} for definition. Let $X^n$ be the solution of the SDE \eqref{eqmainr} with drift $b_n: = b_{1,n} + b_2$.
	Then by Lemma \ref{lemmainres1r}, for every $n\in \mathbb{N}$ it holds
	$$
		E \left[ | D_t X^{n}_s - D_{t'} X^{n}_s |^2 \right] \leq {\cal C}(\|\tilde{b}_1\|_{\infty},|x|^2,b_2^{\text{power}}) |t -t'|^{\alpha}
	$$
	for $0 \leq t' \leq t \leq T_1$, $\alpha = \alpha(s) > 0$ and
	$$
 		\sup_{0 \leq t \leq T_1} E \left[ | D_t X^{n}_s |^2 \right] \leq {\cal C}(\|\tilde{b}_1\|_{\infty},|x|^2, b^{\text{power}}_2),
	$$
	where the r.h.s. do not depend on $n$.
	In particular, $D^i_tX^n_s$ is bounded in $L^2(\Omega, P)$.
	Thus, it follows from Corollary \ref{maincor} and \cite[Lemma 1.2.3]{Nua06} that (up to a subsequence) $(D^i_tX^n_s)_n$ converges to $D^i_tX_s$ in the weak topology of $L^2(\Omega, P)$.
	Since the function $L^2\ni A\mapsto E[|A|^2]$ is convex and lower semicontinuous, it is weakly lower semicontinuous.
	Thus, taking limits in $n$ on both sides above gives the results.
\end{proof}

\subsection{Representation and moment bounds for the Malliavin derivative}
In this subsection we give an explicit representation of the Malliavin derivative $DX^x$ of the solution $X^x$ of the SDE \eqref{eqmain1r}.
To that end, we will assume that the random part $b_2$ of the drift does not depend on $x$.
Such representation can be very useful to derive results concerning $DX^x$.
The representation we obtain will be given in terms of the time-space local time studied in details in \cite{Ein2000}.

In order to define the local time-space integral with respect to $L^{X^x}(t,z)$, we first start by introducing the space $\left(\mathcal{H}^x,\left\|\cdot\right\|_x\right)$  of functions $f:[0,T]\times \mathbb{R}\rightarrow \mathbb{R} $ with the norm
\begin{align}\label{eqsnormlocalta1}
 	\left\|f\right\|_x&:=2\Big(\int_0^1\int_{\mathbb{R}}f^2(s,z)\exp(-\frac{|z-x|^2}{2s})\frac{\diffns s \diff z}{\sqrt{2\pi s}}\Big)^{\frac{1}{2}}\notag
+\int_0^1\int_{\mathbb{R}}|z-x| |f(s,x)|\exp(-\frac{|z-x|^2}{2s})\frac{\diffns s \diff z}{s\sqrt{2\pi s}}.
\end{align}
See for example \cite{Ein2000}. Endowed with this normed,  $\left(\mathcal{H}^x,\left\|\cdot\right\|_x\right)$ is a Banach space. 
It follows from \cite[Lemma 2.7 and Definition 2.8]{BMBPD17} that the local time-space integral of $f\in \mathcal{H}^x$ with respect to $L^{X^x}(t,z)$ is well defined and we have 
\begin{align}\label{eqslocalt11}
	\int_0^t\int_{\mathbb{R}}f(s,z) L^{X^x}(\diffns s,\diffns z)=			\int_0^T\int_{\mathbb{R}}f(s,z) I_{[0,t]}(s)L^{X^x}(\diffns s,\diffns z).
\end{align}
Let us point out that as already observed in \cite[Remark 2.9]{BMBPD17},  functions $f:[0,T]\times \mathbb{R}\rightarrow \mathbb{R} $ of spacial linear growth uniformly in $t$ belong to $\mathcal{H}^x$ and thus the above local space-time integral exists for $x\in \mathbb{R}$. 
 						
We will also need the following representation which will play a key role in our argument (see for example \cite[Lemma 2.11]{BMBPD17}).
Let $f \in \mathcal{H}^x$ be Lipschitz continuous in space. Then for all $t\in [0,T]$, it holds	 
\begin{align}\label{eqtransLT1}
	\int_0^tf'(s,X^x_s)\diffns s=-\int_0^t\int_{\mathbb{R}}f(s,z) L^{X^x}(\diffns s,\diffns z).
\end{align}
Moreover, the local time-space integral of $f \in {\cal H}^0$ admits the decomposition (see the proof of \cite[Theorem 3.1]{Ein2000})
\begin{align}\label{eqslocalt211}
	\int_0^t\int_{\mathbb{R}}f(s,z) L^{B^x}(\diffns s,\diffns z)&=\int_0^t f (s,B^{x}_s)\diffns B_s+\int_{T-t}^T f (T-s,\widehat{B}^{x}_s)\diffns \widetilde{W}_s\notag
 	-\int_{T-t}^T f (T-s,\widehat{B}^x_s)\frac{\widehat{B}_{s}}{T-s}\diffns s,
\end{align}
$0\leq t\leq T$, a.s.,  $B^x$ is the Brownian motion started at $x$ and $\widehat{B}$ is the time-reversed Brownian motion, that is
\begin{align}\label{eqstimrevbm1}
	\widehat{B}_t:=B_{T-t},\,\,0\leq t\leq T.
\end{align}
Further, the process $\widetilde{W}_t,\,\,\,0\leq t\leq T$, is an independent Brownian motion with respect to the filtration $\mathcal{F}_t^{\widehat{B}}$ generated by $\widehat{B}_t$, and satisfies:
\begin{align}\label{eqstimrevbm2}
	\widetilde{W}_t= \widehat{B}_t-B_T+\int_t^T\frac{\widehat{B}_s}{T-s}\diffns s.
\end{align}
 						
We are now ready to give an explicit representation of the  Malliavin derivative of the unique strong solution to the SDE \eqref{eqmain1r} in terms of a local time integral.

\begin{theorem}\label{Thmexplimalder}
	Assume that the conditions of Theorem \ref{thmainres1r} are satisfied. Assume in addition that $b_2$ does not depend on $x$ and that
	\begin{equation}
	\label{eq: power moment b21}
		\sup_{0\le s\le T}E\left[\left(\int_0^T|D_sb_2(t,\omega+\varphi) |^2\diffns t \right)^2 \right]<\infty
	\end{equation}
for every		$\varphi \in C^1_b([0,T],\mathbb{R}^d)$.
	For every $0\leq s\leq t\leq T$, the $i$-th component of the Malliavin derivative of the unique strong solution to the SDE \eqref{eqmain1r} admits the following representation:
 	\begin{align}\label{mallexpliform}
 		D^i_tX_s^x=e^{\int_t^s\int_{\mathbb{R}} b_{1}(u,z) L^{X^x}(\diffns s,\diffns z) }\left(\int_t^sD^i_tb_2(u,\omega)e^{-\int_t^u\int_{\mathbb{R}} b_{1}(r,z) L^{X^x}(\diffns r,\diffns z)}\diffns u + \sigma_i \right),\quad i =1,\dots,d,
 	\end{align}
 	$L^{X^x}(\diffns s,\diffns z)$ is the integration with respect to the time-space local time of $X^x$. 
\end{theorem}
Before proving the above theorem we will need some auxiliary results. The next one generalizes \cite[Lemma A.2]{BMBPD17} to the case where the integrand is of spatial linear growth.  
\begin{remark}
	When the drift $b_2$ does not depend on $x$, the SDE \eqref{eq:sde into} can be solved without use of Malliavin calculus.
	In fact, by a change of measures, the SDE reduces to an equation with deterministic drift, driven by a new Brownian motion (solved e.g. in \cite{MeMo17}).
	Then, one can conclude using the fact that the filtration of the new Brownian motion is smaller than the initial filtration.
\end{remark}
\begin{example}
	Let $\alpha\in {\cal H}^4(\mathbb{R}^d)$ be such that $b_2(t,\omega):=\int_0^t\alpha_s\diffns B_s$ satisfies the moment condition \eqref{eq:expo moment b2} (which is automatically satisfied when $\alpha$ is deterministic), and $\alpha$ H\"older-continuous in $t$, and Malliavin differentiable with $(D_s\alpha_t)_t \in {\cal H}^4(\mathbb{R}^d)$.
	Then, it follows by \citet[Proposition 1.3.4]{Nua06} that $b_2$ is Malliavin differentiable, and
	\begin{equation}
		D_sb_2(t) = \alpha_s1_{\{s\le t \}} + \int_s^tD_s\alpha_r\diffns B_r.
	\end{equation}
	Thus, by Burkholder-Davis-Gundy inequality, it holds
	$$
		\sup_{0\le s\le T}E\left[\left(\int_0^T|D_sb_2(t,\omega)|^2\diffns t\right)^4 \right]\le \|\alpha\|_{{\cal H}^4(\mathbb{R}^d)} + \sup_{0\le s\le T}\|D_s\alpha \|^4_{{\cal H}^4(\mathbb{R}^d)}<\infty,
	$$
	which shows that $b^{\textrm{power}}_2<\infty$.
\end{example}

\begin{remark}
	Let us notice that using Cameron-Martin-Girsanov theorem, one can show that the bound \eqref{eq: power moment b21} holds if, for $\varepsilon >0$ small, we have
		\begin{equation}
		\sup_{0\le s\le T}E\left[\int_0^T|D_sb_2(t,\omega) |^{4+\varepsilon}\diffns t  \right]<\infty.
		\end{equation}
\end{remark}
\begin{lemma}\label{lemmexpk}
	Let $f:[0,T]\times \mathbb{R}\rightarrow \mathbb{R} $ be of spacial linear growth uniformly in $t$. Then for every $t\in[0,T], k\in \mathbb{R}$ and subset $K \subset \mathbb{R}$, it holds
 	$$
 		\sup_{x\in K}E\left[\exp\left(k\int_0^t\int_{\mathbb{R}}f(s,z)L^{B^x}(\diffns s,\diffns z)\right)\right]<\infty,
 	$$
 	provided that $T$ is small enough. In the above, $L^{B^x}(\diffns s,\diffns z)$ is the integration with respect to the time-space local time of $B^x$.
\end{lemma}
\begin{proof}
	It follows from \eqref{eqslocalt211} and the H\"older inequality that 
 	\begin{align}	
 		&E\left[\exp\left(k\int_0^t\int_{\mathbb{R}}f(s,z)L^{B^x}(\diffns s,\diffns z)\right)\right]\\
 		& \leq E\left[\exp\left(2k\int_0^tf(s,B^x_s)\diffns B^x_s\right)\right]^{\frac{1}{2}}
	 	\times E\left[\exp\left(4k\int_{T-t}^Tf(T-s,B^x_{T-s})\diffns \widetilde{W}^x_s\right)\right]^{\frac{1}{4}}\notag\\
 		&\quad\times	E\left[\exp\left(-4k\int_{T-t}^Tf(T-s,B^x_{T-s})\frac{B_{T-s}}{T-s}\diffns s\right)\right]^{\frac{1}{4}}
 		=I_1\times I_2\times I_3.
	\end{align}
 								
	Let us consider $I_1$. Using H\"older inequality, we have 
	\begin{align*}
		E\left[\exp\left(2k\int_0^tf(s,B^x_s)\diffns B^x_s\right)\right]\leq 	E\left[\mathcal{E}\left(4k\int_0^Tf(s,B^x_s)\diffns B^x_s\right)\right]^{\frac{1}{2}} E\left[\exp\left(8k^2\int_0^Tf^2(s,B^x_s)\diffns s\right)\right]^{\frac{1}{2}}.
	\end{align*}
	The Girsanov theorem applied to the martingale $2k\int_0^tf(s,B^x_s)\diffns B^x_s$ yields that the first term in \eqref{mallderI1r} is equal to one. 
	Similar arguments as in the proof of Lemma \ref{lemmainres1r} (i.e. using power series expansion of the exponential function) enables to conclude that the second term above is finite for small $T$. 
 									
	Next we wish to study the boundedness $I_3$. It was already shown in \cite[Lemma A.2]{BMBPD17} that 
	$$
 		E\left[\exp\left(k\int_{0}^T\frac{|B_{T-s}|}{T-s}\diffns s\right)\right]<\infty.
 	$$
 	Hence to show the boundedness of $I_3$, it suffices to show that
 	$$
 		E\left[\exp\left(k\int_{0}^T\frac{|B_{T-s}|^2}{T-s}\diffns s\right)\right]<\infty
 	$$
 	for $T$ small enough. Indeed, using exponential expansion, and the H\"older inequality, we have
 	\begin{align*}
 		E\left[\exp\left(k\int_{0}^T\frac{|B_{T-s}|^2}{T-s}\diffns s\right)\right]&=\sum_{n=1}^{\infty}\frac{1}{n!}E\left[\left(k\int_{0}^T\frac{|B_{T-s}|^2}{T-s}\diffns s\right)^n\right]\notag
 		\leq  \sum_{n=1}^{\infty}\frac{k^n}{n!}\int_{0}^T\frac{E|B_{T-s}|^{2n}}{(T-s)^n}\diffns s\times T^{n-1}\notag\\
 		= & \sum_{n=1}^{\infty}\frac{k^n}{n!}\frac{(2n)!}{2^nn!}T\times T^{n-1}\notag
 		=  \sum_{n=1}^{\infty}(Tk)^n\frac{(2n)!}{2^n(n!)^2}\notag.
 	\end{align*}
 	Using once more the ratio test, one deduces that the above sum is finite for small $T$. Combining arguments for the bounds of $I_1$ and $I_3 $ enables to conclude that $I_2$ is bounded as well.
\end{proof}
\begin{proof}[of Theorem \ref{Thmexplimalder}]
	Let $b_{1,n}$ be a sequence of smooth drifts approximating $b_1$.
	Then, using \eqref{MalliavinDerivativeEquationr} and \eqref{eqtransLT1}, we have  
 	\begin{align} \label{qede2c}
 		D^i_tX^{x,n}_s &= e^{-\int_t^s \int_{\mathbb{R}}b_{1,n}(u,z) L^{X^{x,n}}(\diffns u,\diffns z)} \left(\int_t^sD^i_tb_2(u,\omega)e^{\int_t^u  \int_{\mathbb{R}}b_{1,n}(r,z) L^{X^{x,n}}(\diffns r,\diffns z)}\diffns u + \sigma_i \right).
 	\end{align}
 	It follows from Corollary \ref{maincor} that $(X^{x,n}_{t})_{n \geq 1}$ is relatively compact in $L^2(\Omega,P)$ and $\|D_sX^{x,n}_t\|_{L^2(P\otimes \diffns t)}$ is uniformly bounded in $n$. Hence by \cite[Lemma 1.2.3]{Nua06}, $(D_sX^{x,n}_t)_{n \geq 1}$ converges weakly to $D_sX^{x}_t$ in $L^2(P;\mathbb{R})$. Thus, in order to conclude, we need to show that $$\left\{e^{-\int_t^s \int_{\mathbb{R}}b_{1,n}(u,z) L^{X^{x,n}}(\diffns u,\diffns z)} \left(\int_t^sD^i_tb_2(u,\omega)e^{\int_t^u  \int_{\mathbb{R}}b_{1,n}(r,z) L^{X^{x,n}}(\diffns r,\diffns z)}\diffns u + \sigma_i \right)\mathcal{E}\left(\int_0^T\dot \varphi_r\diffns B_r\right)\right\}_{n}$$ 
 	converges to 
 	$$e^{-\int_t^s \int_{\mathbb{R}}b_{1}(u,z) L^{X^{x}}(\diffns u,\diffns z)} \left(\int_t^sD^i_tb_2(u,\omega)e^{\int_t^u  \int_{\mathbb{R}}b_{1}(r,z) L^{X^{x}}(\diffns r,\diffns z)}\diffns u + \sigma_i \right)\mathcal{E}\left(\int_0^T\dot\varphi_r\diffns B_r\right)$$ 
 	in expectation for every $\varphi \in C^1_b([0,T],\mathbb{R}^d)$. 
 	We will only show that 
 	$$\left\{e^{-\int_t^s \int_{\mathbb{R}}b_{1,n}(u,z) L^{X^{x,n}}(\diffns u,\diffns z)} \int_t^sD^i_tb_2(u,\omega)e^{\int_t^u  \int_{\mathbb{R}}b_{1,n}(r,z) L^{X^{x,n}}(\diffns r,\diffns z)}\diffns u \mathcal{E}\left(\int_0^T\dot\varphi_r\diffns B_r\right)\right\}_{n}$$
 	 converges to $e^{-\int_t^s \int_{\mathbb{R}}b_{1}(u,z) L^{X^{x}}(\diffns u,\diffns z)} \int_t^sD^i_tb_2(u,\omega)e^{\int_t^u  \int_{\mathbb{R}}b_{1}(r,z) L^{X^{x}}(\diffns r,\diffns z)}\diffns u \mathcal{E}\Big(\int_0^T\dot\varphi_r\diffns B_r\Big)$ in expectation.
 	Using Girsanov theorem and the Cameron-Martin theorem, we have
	\begin{align}
 	L=	&	\Big|E\Big[\mathcal{E}\Big(\int_0^T\dot\varphi_r\diffns B_r\Big)
 		\Big\{e^{-\int_t^s \int_{\mathbb{R}}b_{1,n}(u,z) L^{X^{x,n}}(\diffns u,\diffns z)} \int_t^sD^i_tb_2(u,\omega)e^{\int_t^u  \int_{\mathbb{R}}b_{1,n}(r,z) L^{X^{x,n}}(\diffns r,\diffns z)}\diffns u\notag\\
 		&	-e^{-\int_t^s \int_{\mathbb{R}}b_{1}(u,z) L^{X^{x}}(\diffns u,\diffns z)} \int_t^sD^i_tb_2(u,\omega)e^{\int_t^u  \int_{\mathbb{R}}b_{1}(r,z) L^{X^{x}}(\diffns r,\diffns z)}\diffns u\Big\}
 		\Big]\Big|\notag\\
 	=&	\Big|E\Big[\mathcal{E}\Big(\int_0^T\dot\varphi_r\diffns B_r\Big)
 		\Big\{e^{-\int_t^s \int_{\mathbb{R}}b_{1,n}(u,z) L^{X^{x,n}}(\diffns u,\diffns z)} \int_t^sD^i_tb_2(u,\omega)e^{\int_t^u  \int_{\mathbb{R}}b_{1,n}(r,z) L^{X^{x,n}}(\diffns r,\diffns z)}\diffns u\notag\\
 		&	-e^{-\int_t^s \int_{\mathbb{R}}b_{1,n}(u,z) L^{X^{x,n}}(\diffns u,\diffns z)} \int_t^sD^i_tb_2(u,\omega)e^{\int_t^u  \int_{\mathbb{R}}b_{1}(r,z) L^{X^{x}}(\diffns r,\diffns z)}\diffns u\notag\\
 		& + e^{-\int_t^s \int_{\mathbb{R}}b_{1,n}(u,z) L^{X^{x,n}}(\diffns u,\diffns z)} \int_t^sD^i_tb_2(u,\omega)e^{\int_t^u  \int_{\mathbb{R}}b_{1}(r,z) L^{X^{x}}(\diffns r,\diffns z)}\diffns u\notag\\
 		&-e^{-\int_t^s \int_{\mathbb{R}}b_{1}(u,z) L^{X^{x}}(\diffns u,\diffns z)} \int_t^sD^i_tb_2(u,\omega)e^{\int_t^u  \int_{\mathbb{R}}b_{1}(r,z) L^{X^{x}}(\diffns r,\diffns z)}\diffns u\Big\}
 		\Big]\Big|\notag\\
 		=&	\Big|E\Big[\mathcal{E}\Big(\int_0^T\dot\varphi_r\diffns B_r\Big)
 		\Big\{e^{-\int_t^s \int_{\mathbb{R}}b_{1,n}(u,z) L^{X^{x,n}}(\diffns u,\diffns z)} \notag\\
 		&\times \int_t^sD^i_tb_2(u,\omega)\Big(e^{\int_t^u  \int_{\mathbb{R}}b_{1,n}(r,z) L^{X^{x,n}}(\diffns r,\diffns z)}-e^{\int_t^u  \int_{\mathbb{R}}b_{1}(r,z) L^{X^{x}}(\diffns r,\diffns z)}\Big)\diffns u\notag\\
 		& + \Big(e^{-\int_t^s \int_{\mathbb{R}}b_{1,n}(u,z) L^{X^{x,n}}(\diffns u,\diffns z)} -e^{-\int_t^s \int_{\mathbb{R}}b_{1}(u,z) L^{X^{x}}(\diffns u,\diffns z)}\Big)\int_t^sD^i_tb_2(u,\omega)e^{\int_t^u  \int_{\mathbb{R}}b_{1}(r,z) L^{X^{x}}(\diffns r,\diffns z)}\diffns u\Big\}
 		\Big]\Big|\notag\\
 				=&	\Big|E\Big[
 				e^{-\int_t^s \int_{\mathbb{R}}b_{1,n}(u,z) L^{\tilde{X}^{x,n}}(\diffns u,\diffns z)} \notag\\
 				&\times \int_t^sD^i_tb_2(u,\omega+\varphi)\Big(e^{\int_t^u  \int_{\mathbb{R}}b_{1,n}(r,z) L^{\tilde{X}^{x,n}}(\diffns r,\diffns z)}-e^{\int_t^u  \int_{\mathbb{R}}b_{1}(r,z) L^{\tilde{X}^{x}}(\diffns r,\diffns z)}\Big)\diffns u\notag\\
 				& + \Big(e^{-\int_t^s \int_{\mathbb{R}}b_{1,n}(u,z) L^{\tilde{X}^{x,n}}(\diffns u,\diffns z)} -e^{-\int_t^s \int_{\mathbb{R}}b_{1}(u,z) L^{\tilde{X}^{x}}(\diffns u,\diffns z)}\Big)\int_t^sD^i_tb_2(u,\omega+\varphi)e^{\int_t^u  \int_{\mathbb{R}}b_{1}(r,z) L^{\tilde{X}^{x}}(\diffns r,\diffns z)}\diffns u
 				\Big]\Big|\notag\\
 				\leq & I_{1,n}+I_{2,n}.
	\end{align}
 			Let us concentrate on $ I_{1,n}$.	Repeated use of H\"older inequality, Girsanov transform, the bound on $D^i_tb_2(u,\omega+\varphi)$ and the fact that $|e^x-1|\leq |x|(e^x+1)$ gives 
 			\begin{align}
 			I_{1,n}\leq &\Big|E\Big[
 			e^{-2\int_t^s \int_{\mathbb{R}}b_{1,n}(u,z) L^{\tilde{X}^{x,n}}(\diffns u,\diffns z)} \Big]^{1/2}\times E\Big[\Big(\int_t^s(D^i_tb_2(u,\omega+\varphi))^2\diffns u\Big)^2\Big]^{1/4}\notag\\
 			&\times E\Big[ \Big(\int_t^s\Big(e^{\int_t^u  \int_{\mathbb{R}}b_{1,n}(r,z) L^{\tilde{X}^{x,n}}(\diffns r,\diffns z)}-e^{\int_t^u  \int_{\mathbb{R}}b_{1}(r,z) L^{\tilde{X}^{x}}(\diffns r,\diffns z)}\Big)^2\diffns u\Big)^2\Big]^{1/4}\Big|\notag\\
\leq & C E\Big[\mathcal{E}\Big(\int_0^T\Big\{\tilde u_n(r,\omega,x+\sigma \cdot B_r)+\dot\varphi_r\Big\}\diffns B_r \Big)e^{-2\int_t^s  \int_{\mathbb{R}}b_{1,n}(r,z) L^{\|\sigma\|B_\sigma^{x}}(\diffns r,\diffns z)}\Big]^{1/2}\notag\\
&\times  \Big(\int_t^sE\Big[\Big(e^{\int_t^u  \int_{\mathbb{R}}b_{1,n}(r,z) L^{\tilde{X}^{x,n}}(\diffns r,\diffns z)}-e^{\int_t^u  \int_{\mathbb{R}}b_{1}(r,z) L^{\tilde{X}^{x}}(\diffns r,\diffns z)}\Big)^4\Big]\diffns u\Big)^{1/4}\notag\\
\leq& C E\Big[\mathcal{E}\Big(\int_0^T\Big\{\tilde u_n(r,\omega,x+\sigma \cdot B_r)+\dot\varphi_r\Big\}\diffns B_r \Big)^2\Big]^{1/4}E\Big[e^{-4\int_t^s  \int_{\mathbb{R}}b_{1,n}(r,z) L^{\|\sigma\|B_\sigma^{x}}(\diffns r,\diffns z)}\Big]^{1/4}\notag\\
&\times  \Big(\int_t^sE\Big[\Big|e^{\int_t^u  \int_{\mathbb{R}}b_{1,n}(r,z) L^{\tilde{X}^{x,n}}(\diffns r,\diffns z)}-e^{\int_t^u  \int_{\mathbb{R}}b_{1}(r,z) L^{\tilde{X}^{x}}(\diffns r,\diffns z)}\Big|^{1/2}\Big]^{1/2}\notag\\
& \times E\Big[\Big|e^{\int_t^u  \int_{\mathbb{R}}b_{1,n}(r,z) L^{\tilde{X}^{x,n}}(\diffns r,\diffns z)}-e^{\int_t^u  \int_{\mathbb{R}}b_{1}(r,z) L^{\tilde{X}^{x}}(\diffns r,\diffns z)}\Big|^{15}\Big]^{1/2}\diffns u\Big)^{1/4}\notag\\
\leq & C E\Big[\mathcal{E}\Big(\int_0^T\Big\{\tilde u_n(r,\omega,x+\sigma \cdot B_r)+\dot\varphi_r\Big\}\diffns B_r \Big)^2\Big]^{1/4}E\Big[e^{-4\int_t^s  \int_{\mathbb{R}}b_{1,n}(r,z) L^{\|\sigma\|B_\sigma^{x}}(\diffns r,\diffns z)}\Big]^{1/4}\notag\\
&\times  \Big(\int_t^sE\Big[\Big|e^{\int_t^u  \int_{\mathbb{R}}b_{1,n}(r,z) L^{\tilde{X}^{x,n}}(\diffns r,\diffns z)}-e^{\int_t^u  \int_{\mathbb{R}}b_{1}(r,z) L^{\tilde{X}^{x}}(\diffns r,\diffns z)}\Big|^{1/2}\Big]^{1/2}\notag\\
& \times E\Big[e^{15\int_t^u  \int_{\mathbb{R}}b_{1,n}(r,z) L^{\tilde{X}^{x,n}}(\diffns r,\diffns z)}+e^{15\int_t^u  \int_{\mathbb{R}}b_{1}(r,z) L^{\tilde{X}^{x}}(\diffns r,\diffns z)}\Big]^{1/2}\diffns u\Big)^{1/4}\notag\\
	\leq & C E\Big[\mathcal{E}\Big(\int_0^T\Big\{\tilde u_n(r,\omega,x+\sigma \cdot B_r)+\dot\varphi_r\Big\}\diffns B_r \Big)^2\Big]^{1/4}E\Big[e^{-4\int_t^s  \int_{\mathbb{R}}b_{1,n}(r,z) L^{\|\sigma\|B_\sigma^{x}}(\diffns r,\diffns z)}\Big]^{1/4}\notag\\
	&\times  \Big(\int_t^sE\Big[\Big|\int_t^u  \int_{\mathbb{R}}b_{1,n}(r,z) L^{\tilde{X}^{x,n}}(\diffns r,\diffns z)-\int_t^u  \int_{\mathbb{R}}b_{1}(r,z) L^{\tilde{X}^{x}}(\diffns r,\diffns z)\Big|^{1/2}\notag\\
	&\times \Big|e^{\int_t^u  \int_{\mathbb{R}}b_{1,n}(r,z) L^{\tilde{X}^{x,n}}(\diffns r,\diffns z)+\int_t^u  \int_{\mathbb{R}}b_{1}(r,z) L^{\tilde{X}^{x}}(\diffns r,\diffns z)}+1\Big|^{1/2}
\Big]^{1/2}\notag\\
	& \times E\Big[e^{15\int_t^u  \int_{\mathbb{R}}b_{1,n}(r,z) L^{\tilde{X}^{x,n}}(\diffns r,\diffns z)}+e^{15\int_t^u  \int_{\mathbb{R}}b_{1}(r,z) L^{\tilde{X}^{x}}(\diffns r,\diffns z)}\Big]^{1/2}\diffns u\Big)^{1/4}\notag
	\end{align}
	\begin{align}
	&\leq  C E\Big[\mathcal{E}\Big(\int_0^T\Big\{\tilde u_n(r,\omega,x+\sigma \cdot B_r)+\dot\varphi_r\Big\}\diffns B_r \Big)^2\Big]^{1/4}E\Big[e^{-4\int_t^s  \int_{\mathbb{R}}b_{1,n}(r,z) L^{\|\sigma\|B_\sigma^{x}}(\diffns r,\diffns z)}\Big]^{1/4}\notag\\
	&\quad \times  \Big(\int_t^sE\Big[\Big|\int_t^u  \int_{\mathbb{R}}b_{1,n}(r,z) L^{\tilde{X}^{x,n}}(\diffns r,\diffns z)-\int_t^u  \int_{\mathbb{R}}b_{1}(r,z) L^{\tilde{X}^{x}}(\diffns r,\diffns z)\Big|\Big]^{1/4}\notag\\
	&\quad\times E\Big[\Big|e^{\int_t^u  \int_{\mathbb{R}}b_{1,n}(r,z) L^{\tilde{X}^{x,n}}(\diffns r,\diffns z)+\int_t^u  \int_{\mathbb{R}}b_{1}(r,z) L^{\tilde{X}^{x}}(\diffns r,\diffns z)}+1\Big|
\Big]^{1/4}\notag\\
	&\quad\times E\Big[e^{15\int_t^u  \int_{\mathbb{R}}b_{1,n}(r,z) L^{\tilde{X}^{x,n}}(\diffns r,\diffns z)}+e^{15\int_t^u  \int_{\mathbb{R}}b_{1}(r,z) L^{\tilde{X}^{x}}(\diffns r,\diffns z)}\Big]^{1/2}\diffns u\Big)^{1/4}\notag.
\end{align}
 	In the above $B^\sigma:=\sum_{i=1}^n\frac{\sigma_i}{\|\sigma\|}B^i$ is a standard Brownian motion.
 					Using Cauchy-Schwartz inequality, the Novikov's condition on $b_2$ and Bene\v{s} Theorem, the first term is finite  for small time $T$. Using Lemma \ref{lemmexpk} and \cite[Proposition 2.1.1]{Ein2006} enables to conclude that the second term is bounded. Using once more Cauchy-Schwartz inequality, Girsanov transform and Lemma \ref{lemmexpk}, one deduces that the fourth and fifth terms are bounded for small time $T$. Let use now focus on the third term. By Girsanov transform and Cauchy-Schwartz inequality, we have 
 	\begin{align}\label{eqrefne1}
 		& E\Big[\int_t^u  \int_{\mathbb{R}}b_{1,n}(r,z) L^{\tilde{X}^{x,n}}(\diffns r,\diffns z)-\int_t^u  \int_{\mathbb{R}}b_{1}(r,z) L^{\tilde{X}^{x}}(\diffns r,\diffns z)\Big]\notag\\
 		&= E\Big[\mathcal{E}\Big(\int_0^T\Big\{\tilde u_n(r,\omega,x+\sigma \cdot B_r)+\dot\varphi_r\Big\}\diffns B_r \Big)\int_t^u  \int_{\mathbb{R}}b_{1,n}(r,z) L^{\|\sigma\|B_\sigma^{x}}(\diffns r,\diffns z)\notag\\
		& \quad-\mathcal{E}\Big(\int_0^T\Big\{\tilde u(r,\omega,x+\sigma \cdot B_r)+\dot\varphi_r\Big\}\diffns B_r \Big)\int_t^u  \int_{\mathbb{R}}b_{1}(r,z) L^{\|\sigma\|B_\sigma^{x}}(\diffns r,\diffns z)\Big] \notag\\
 			&= E\Big[\mathcal{E}\Big(\int_0^T\Big\{\tilde u_n(r,\omega,x+\sigma \cdot B_r)+\dot\varphi_r\Big\}\diffns B_r \Big)\int_t^u  \int_{\mathbb{R}}\Big(b_{1,n}(r,z) -b_{1}(r,z)\Big)L^{\|\sigma\|B_\sigma^{x}}(\diffns r,\diffns z)\notag\\
 			&\quad+ \Big\{\mathcal{E}\Big(\int_0^T\Big\{\tilde u_n(r,\omega,x+\sigma \cdot B_r)+\dot\varphi_r\Big\}\diffns B_r \Big)-\mathcal{E}\Big(\int_0^T\Big\{\tilde u(r,\omega,x+\sigma \cdot B_r)+\dot\varphi_r\Big\}\diffns B_r \Big)\Big\}\int_t^u  \int_{\mathbb{R}}b_{1}(r,z) L^{\|\sigma\|B_\sigma^{x}}(\diffns r,\diffns z)\Big]\notag\\
 			&\leq   E\Big[\mathcal{E}\Big(\int_0^T\Big\{\tilde u_n(r,\omega,x+\sigma \cdot B_r)+\dot\varphi_r\Big\}\diffns B_r \Big)^2\Big]^{1/2}E\Big[\Big(\int_t^u  \int_{\mathbb{R}}\Big(b_{1,n}(r,z) -b_{1}(r,z)\Big)L^{\|\sigma\|B_\sigma^{x}}(\diffns r,\diffns z)\Big)^2\Big]^{1/2}\notag\\
 			&\quad  + E\Big[\Big\{\mathcal{E}\Big(\int_0^T\Big\{\tilde u_n(r,\omega,x+\sigma \cdot B_r)+\dot\varphi_r\Big\}\diffns B_r \Big)-\mathcal{E}\Big(\int_0^T\Big\{\tilde u(r,\omega,x+\sigma \cdot B_r)+\dot\varphi_r\Big\}\diffns B_r \Big)\Big\}^2\Big]^{1/2}\notag\\
 		&\quad	\times  E\Big[\Big(\int_t^u  \int_{\mathbb{R}}b_{1}(r,z) L^{\|\sigma\|B_\sigma^{x}}(\diffns r,\diffns z)\Big)^2\Big]^{1/2}.
 			 \end{align}
	Using  Lemma \ref{lemmexpk}, the first and the last terms on the right of \eqref{eqrefne1} is bounded. Using \eqref{eqslocalt211}, Minkowski, Doob maximal inequality and the dominated convergence theorem the second term converges to $0$. Similar reasoning as in the proof of Lemma \ref{lem:weak conv weak sol} enable to conclude that the third term converges to $0$. Thus the result follows.	
\end{proof}

\section{Stochastic differentiable flow for SDEs with random drifts}
The aim of this section is to prove existence of a Sobolev differentiable stochastic flow for the SDE \eqref{eq:sde into} with (non-Markovian) random drifts.
Due to the additive decomposition assumption $b (t,\omega,x) = b_1(t,x) + b_2(t, \omega)$, the analysis of the flow turns out to be much easier than that of the Malliavin derivative considered above.
Most of the result of this section will follow as adaptations of some arguments of \citet{MNP2015}.

Throughout this section, we denote $b_n: = b_{1,n} + b_2$ where $b_{1,n}:[0,T]\times \mathbb{R}\to \mathbb{R}$, is the sequence of smooth functions with compact support approximating $b_1$ as introduced in the proof of Theorem \ref{thmainres1r}, see Step 3.2.2.
We further denote by $X^{s,x,n}$ the unique solution of the SDE associated to $b_n$ with initial condition $X^{s,x,n}_s = x$.
That is,
\begin{equation*}
	X^{s,x,n}_t = x + \int_s^tb_{n}(u, X^{s,x,n}_u,\omega)\diffns u + \sigma\cdot B_t.
\end{equation*}
We first prove differentiability of the solution $X^{s,x,n}$, and derive a uniform (in $n$) bound on its deriviative.
\begin{proposition}
\label{pro:bound derivative}
	Let $p\ge1$ and $(s,x) \in [0,T]\times \mathbb{R}$ be fixed.
	If $T$ is small enough, then for every $t \in [0,T]$, almost every trajectories of $x \mapsto X^{s,x,n}_t$ is differentiable and it holds
	% \begin{equation}
	% \label{eq:bound derivative}
	% 	\sup_{n \in \mathbb{N}}E\left[\left|\frac{\partial}{\partial x}X^{s,x,n}_t\right|^p \right]\le {\cal C}_p(||\tilde b_1||_\infty, |x|^2, T)
	% \end{equation}
	\begin{align}\label{eqapprosobder1}
		E\left[|\partial_x X^{s,x,n}_t|^p \right] &\le \tilde{C}_{p,T,\sigma,\|\tilde{b}\|_\infty} \exp\left(C_{p,T,\sigma,\|\tilde{b}\|_\infty}|x|^{2}\right)
	\end{align}
	for some positive constants $\tilde{C}_{p,T,\sigma,\|\tilde{b}\|_\infty}$ and $C_{p,T,\sigma,\|\tilde{b}\|_\infty}$ depending on $p,T,\sigma$ and $\|\tilde{b}\|_\infty$.
%	for some continuous function ${\cal C}_p$ that is positive and increasing in each coordinates.
\end{proposition}
\begin{proof}
	The differentiability of the trajectories of $x \mapsto X^{s,x,n}_t$ follows from the seminal work \cite{Kun90}, from which we further obtain that $\partial_xX^{s,x,n}_u:= \frac {\partial }{\partial x}X^{s,x,n}_t$ satisfies
	\begin{equation}
		\partial_xX^{s,x,n}_t = 1 + \int_s^t\left(b'_{1,n}(u, X^{s,x,n}_u) + b_2'(u, X^{s,x,n}_u,\omega)\right)\partial_xX^{s,x,n}_u\diffns u,
	\end{equation}
	where $b'_{1,n} = \frac{\partial}{\partial x}b_{1,n}$.
	The solution of this (random) ODE can be explicitly given by
	\begin{equation*}
		\partial_xX^{s,x,n}_t = \exp\left(\int_s^tb_{1,n}'(u, X^{s,x,n}_u) + b_2'(u, X^{s,x,n}_u,\omega) \diffns u \right).
	\end{equation*}
	Thus, Girsanov's theorem and successive applications of H\"older's inequality give (recall the definition of $u_n$ given in \eqref{eqgirs1 n})
	\begin{align*}
		E\left[|\partial_x X^{s,x,n}_t|^p \right] &\le E\left[\exp\left(\int_s^t2pb'_{1,n}(u,X^{s,x,n}_u)\diffns u\right) \right]^{1/2}E\left[\exp\left(\int_s^t 2pb_2'(u, X^{s,x,n}_u,\omega)\diffns u\right) \right]^{1/2}\\
		 &\le E\left[e^{2pTM_2} \right]^{1/2}E\left[ {\cal E}\left(\int_0^tu_{n}(u, x + \sigma\cdot B_u ,\omega)\diffns B_u \right)\exp\left(2p\int_s^t b'_{1,n}(u, x + \sigma\cdot B_u)\right) \right]^{1/2}\\
		&\le CE\left[{\cal E}\left(4\int_0^Tu_n(u, x + \sigma \cdot B_u, \omega)\diffns B_u \right) \right]^{1/8}
		E\left[ \exp\left(6\int_0^T|u_n(u, x+ \sigma\cdot B_u, \omega)|^2\diffns u\right) \right]^{1/8}\\
		&\quad \times E\left[\exp\left(4p\int_s^tb'_{1,n}(u, x  +\sigma\cdot B_u)\diffns u \right) \right]^{1/4} =: I_1^n\times I^n_2\times I^n_3.
	\end{align*}
	Since $b_{1,n}$ is of linear growth and $b_2$ square integrable, it holds $I^n_1 = 1$.
	As shown in the proof of Lemma \ref{lemmainres1r}, if $T$ is small enough, the sequence $I_2^n$ is bounded, and as shown in \eqref{eqapenexpobound3}, we have
	\begin{align*}
		I^n_3 &=\sum_{q=1}^\infty\frac{E\left[|\int_s^t(4p)b'_{1,n}(u, x+\sigma\cdot B_u)\diffns u|^q \right]}{q!}
		\leq  \sum_{q=1}^\infty\frac{(4p)^{q}E\left[|\int_s^tb'_{1,n}(u, x+\sigma\cdot B_u)\diffns u|^{2q} \right]^{\frac{1}{2}}}{q!}\\
		&\leq   \sum_{q=1}^\infty\frac{(4p)^{q}(C_\sigma)^q(1+|x|^{2q})^{1/2}\|\tilde{b}\|_\infty^q|t-s|^q(q!)^{1/2}}{q!}
		\leq   \sum_{q=1}^\infty\frac{2^q(4p)^{q}(C_\sigma)^q(1+|x|^{q})\|\tilde{b}\|_\infty^q|t-s|^q}{2^q(q!)^{1/2}}\\
		&\leq  \left(\sum_{q=1}^\infty\frac{2^{2q}(4p)^{2q}(C_\sigma)^{2q}(1+|x|^{q})^{2}\|\tilde{b}\|_\infty^{2q}|t-s|^{2q}}{q!}\right)^{\frac{1}{2}} \left(\sum_{q=1}^\infty\frac{1}{2^{2q}}\right)^{\frac{1}{2}}\\
		&\leq 2\left(\exp\left(16p^2(C_\sigma)^{2}\|\tilde{b}\|_\infty^{2}|t-s|^2(1+|x|^{2})\right)\right)^{\frac{1}{2}}\\
		&\leq 2\exp\left(p^2C_{\sigma,\|\tilde{b}\|_\infty^{2}}T(1+|x|^{2})\right).
	\end{align*}
	This concludes the proof that is there exist positive constants $\tilde{C}_{p,T,\sigma,\|\tilde{b}\|_\infty}$ and $C_{p,T,\sigma,\|\tilde{b}\|_\infty}$ depending on $p,T,\sigma$ and $\|\tilde{b}\|_\infty$ such that 
\begin{align*}%\label{eqapprosobder1}
	E\left[|\partial_x X^{s,x,n}_t|^p \right] &\le \tilde{C}_{p,T,\sigma,\|\tilde{b}\|_\infty} \exp\left(C_{p,T,\sigma,\|\tilde{b}\|_\infty}|x|^{2}\right).
\end{align*}
	
\end{proof}
\begin{corollary}
	Let $p\ge 2$. 
	If $T$ is small enough, and $M_2$ in \ref{a2} is bounded, then it holds
	\begin{equation*}
		E\left[ |X^{s_1, x_1}_t - X^{s_2, x_2}_t |^p\right]\le {\cal C}_p(||\tilde b_1||_\infty,|x|^2, T)\left(|s_1 - s_2|^{p/2} + |x_1 - x_2|^{p} \right)
	\end{equation*}
	for every $s_1, s_2, t \in [0,T]$, $x_1, x_2 \in \mathbb{R}$ and for some continuous function ${\cal C}_p$ increasing in each component.

	In particular, for every $t\in [0,T]$ almost every trajectories\footnote{We use the convention $X^{s,x}_t = x$ whenever $t\le s$.} of $(s,x)\mapsto X^{s,x}_t$ is $\alpha$-H\"older continuous with $\alpha <1/2$ in $s$ and $\alpha <1$ in $x$.
\end{corollary}
\begin{proof}
	Assume without loss of generality that $s_1\le s_2 <t$.
	For ease of notation, put $X^{i,n}:= X^{s_i,x_i,n}$, $i=1,2$.
	Then, for every $n \in \mathbb{N}$, we have
	\begin{align*}
		X^{1, n}_t - X^{2, n}_t & = x_1 - x_2 + \int_{s_1}^tb_{1, n}(u, X^{1,n}_u) + b_2(u,X^{1,n}_u, \omega)\diffns u + \int_{s_1}^t\sigma\cdot \diffns B_u\\
		&\quad - \int_{s_2}^tb_{1, n}(u, X^{2,n}_u) + b_2(u, X^{2,n}_u,\omega)\diffns u - \int_{s_2}^t\sigma\cdot \diffns B_u\\
		& = x_1 - x_2 + \int_{s_1}^{s_2}b_{1,n}(u, X^{1, n}_u) + b_{2}(u,X^{1,n}_u,\omega) \diffns u + \int_{s_2}^tb_{1,n}(u,X^{1,n}_u) - b_{1,n}(u, X^{2,n}_u)\diffns u\\
		&\quad +\int_{s_2}^tb_{2}(u,X^{1,n}_u,\omega) - b_{2}(u, X^{2,n}_u,\omega)\diffns u + \sigma\cdot B_{s_1} - \sigma\cdot B_{s_2}.
	%	&\le x_1 - x_2 + \left(\int_{0}^{T}|b_{1,n}(u, X^{1, n}_u) + b_{2}(u,X^{1,n}_u,\omega)|^2 \diffns u\right)^{1/2}|s_1 - s_2|^{1/2}\\
	%	&\quad + \int_{s_2}^tb_{1,n}(u,X^{1,n}_u) - b_{1,n}(u, X^{2,n}_u)\diffns u + \int_{s_2}^tb_{2}(u,X^{1,n}_u,\omega) - b_{2}(u, X^{2,n}_u,\omega)\diffns u+\sigma\cdot B_{s_1} - \sigma\cdot B_{s_2}.
	\end{align*}
	Using H\"older continuity of the paths of Brownian motion and the mean-value theorem applied to $x\mapsto b_{1,n}(u,X^{s_1,x}_u)$, this allows to obtain the following estimates:
	\begin{align}
		&\nonumber E\left[ |X^{1, n}_t - X^{2, n}_t|^p\right] \le C_p\left\{|x_1 - x_2|^p + |s_1 - s_2|^{ p/2} +  E\left[\left(\int_{0}^{T}\right.|b_{1,n}(u, X^{1, n}_u) + b_{2}(u,X^{1,n}_u,\omega)|^2 \diffns u\right)^{p/2}\right]|s_1 - s_2|^{p/2}\\
		\nonumber &\quad + |x_1 - x_2|^p E\left[\Big|\left.\int_{s_2}^t\int_0^1 b'_{1,n}(u,X^{s_1,x_1+ \tau(x_2 - x_1),n}_u)\partial_xX^{s_1,x_1+ \tau(x_2 - x_1),n}_u \diffns \tau\diffns u\Big|^p\right] \right\}+CE\left[\int_{s_2}^t|X^{1,n}_u - X^{2,n}_u|^p\diffns u\right]\\
		\nonumber &\le C_p\left\{|x_1 - x_2|^p + |s_1 - s_2|^{ p/2} +  E\left[\left(\int_{0}^{T}\right.|b_{1,n}(u, X^{1, n}_u) + b_{2}(u,X^{1,n}_u,\omega)|^2 \diffns u\right)^{p/2}\right]|s_1 - s_2|^{p/2}\\\notag
		\nonumber &\quad + |x_1 - x_2|^p \left.\int_0^1E\left[\Big||\partial_xX^{s_1,x_1+ \tau(x_2 - x_1),n}_{t} - \partial_xX^{s_1,x_1+ \tau(x_2 - x_1),n}_{s_2}\Big|^p\right]\diffns  \tau \right\} +CE\left[\int_{s_2}^t|X^{1,n}_u - X^{2,n}_u|^p\diffns u\right]\\
		\nonumber &\le C_p\left(|x_1 - x_2|^p + |s_1 - s_2|^{ p/2}\right) \left\{E\left[\left(\int_{0}^{T}\right.|b_{1,n}(u, X^{1, n}_u) + b_{2}(u,X^{1,n}_u,\omega)|^2 \diffns u\right)^{p/2}\right]\\
		&\quad + \left.\sup_{t\in [0,T];|x|\le |x_1|+|x_2|}\sup_{n\in \mathbb{N}}E\left[\Big| \partial_xX^{s_1,x,n}_{t}\Big|^p\right]\right\}+C\int_{s_2}^tE\left[|X^{1,n}_u - X^{2,n}_u|^p\right]\diffns u
		\label{eq:intermed bound dx}.
	\end{align}
	Since $|b_{1,n}(t,x)|\le ||\tilde b_1||_\infty(1  + |x|)$ and $b_2$ has exponential moments (see \ref{a1}) it follows by Jensen's inequality that
	\begin{align*}
		E\left[\left(\int_{0}^{T}|b_{1,n}(u, X^{1, n}_u) + b_{2}(u,\omega)|^2 \diffns u\right)^{p/2 }\right] &\le C_{p}T\left(||\tilde b_1||_\infty E\left[\int_0^T1+|X^{1,n}_u|^{p/2}\,du \right] + b_2^{\mathrm{exp}} \right)\\
		&\le C_pT^2\left(||\tilde b_1||_\infty(1 + \sup_nE\left[|X^{1,n}_t|^{p/2} \right]) + b_2^{\mathrm{exp}} \right).
	\end{align*}
	Since by Lemma \ref{lem:weak conv weak sol} it holds $\sup_nE[|X_t^{1,n}|^{p/2}]<\infty$, it follows by \eqref{eq:intermed bound dx} (after application of Gronwall's inequality) and Proposition \ref{pro:bound derivative} that
	\begin{equation}
	\label{eq:last intermediate dx}
		E\left[ |X^{1, n}_t - X^{2, n}_t|^p\right]\le{\cal C}_p(||\tilde b_1||_\infty, T)\left(|s_1- s_2|^{p/2} + |x_1 - x_2|^p \right).
	\end{equation}
	Let $i=1,2$.
	Since $(X^{s_i, x_i, n}_t)$ converges weakly to the unique solution $X^{s_i,x_i}_t$ of the SDE \eqref{eq:sde into} with drift $b$, (see Lemma \ref{lem:weak conv weak sol} and Theorem \ref{thm:adapted}) it follows by convexity and lower-semicontinuity of $K\mapsto E[|K|^p]$ that, taking the limit in \eqref{eq:last intermediate dx} yields the desired result.
\end{proof}

We conclude this section with the proof of Theorem \ref{thm:flow}.
\begin{proof}[of Theorem \ref{thm:flow}]
	Let $p\ge2$.
	In other to show that $x\mapsto X^{s,x}$ is weakly differentiable, we start by showing that the sequence $(\partial_xX^{s,x,n})_n$ is bounded in $L^2(\Omega, L^p(\mathbb{R},w))$.
	In fact, it follows by H\"older's inequality and Proposition \ref{pro:bound derivative} that
	\begin{align}
			E\left[\left(\int_\mathbb{R}|\partial_xX^{s,x,n}_t|^pw(x)\diffns x \right)^{2/p}\right]&\le E\left[\int_\mathbb{R}|\partial_xX^{s,x,n}_t|^pw(x)\diffns x \right]^{2/p}\notag\\
			&= \left(\int_\mathbb{R}E\left[|\partial_xX^{s,x,n}_t|^p\right]w(x)\diffns x \right)^{2/p}\notag\\
			&\leq  \tilde{C}_{p,T,\sigma,\|\tilde{b}\|_\infty} \left(\int_\mathbb{R}\exp\left(C_{p,T,\sigma,\|\tilde{b}\|_\infty}|x|^{2}\right)w(x)\diffns x \right)^{2/p}<\infty.
			\label{eq:estimate weak derivative1}
		\end{align} 
		%\textcolor{blue}{Note that if we choose for example $w(x)=e^{-x^4}$ the above bound holds.}
		
	Thus, $(\partial_xX^{s,x,n}_t)$ admits a weakly converging subsequence $(\partial_xX^{s,x,n_k}_t)$ in $L^2(\Omega, L^p(\mathbb{R},w))$ to a limit $Y^{s,x}_t$.
	Thus, for every $A \in {\cal F}$ and $f \in L^q(\mathbb{R},w)$ (where $q$ is the Sobolev conjugate of $p$), we have
	\begin{equation*}
		\lim_{k\to \infty}E\left[1_A\int\frac{\partial}{\partial x}X^{n_k,x}_tf(x)w(x)\,dx \right] = E\left[ 1_A\int Y^{s,x}_tf(x)w(x)\,dx\right].
	\end{equation*}
	Choosing $f$ such that $fw \in L^q(\mathbb{R},dx)$, it follows that $Y^{s,x}_t$ is the weak derivative of $X^{s,x}_t$.
	By weak convergence and \eqref{eqapprosobder1},
	it follows that
	\begin{equation*}
		E\left[\left(\int_\mathbb{R}|\partial_xX^{s,x}_t|^pw(x)\diffns x \right)^{2/p}\right]<\infty,
	\end{equation*}
	where $\partial_xX^{s,x}_t$ denotes the weak derivative of $X^{s,x}_t$.

	It remains to show that
	\begin{equation}
	\label{eq:last bound}
		E\left[\left(\int_\mathbb{R}|X^{s,x}_t|^pw(x)\diffns x \right)^{2/p}\right]
		<\infty.
	\end{equation}
	By Jensen's inequality and the linear growth of $b_1$, we have
	\begin{align*}
		&\int_\mathbb{R}|X^{s,x}_t|^pw(x)\,dx\\
		 &\quad\le C_{p}\left(\int_\mathbb{R}|x|^pw(x)\diffns x + t^{p-1}\int_\mathbb{R}\int_0^t(1 + |X^{s,x}_u|^p)w(x)\diffns u\diffns x + t^{p-1}\left\{\int_0^t|b_2(u,X^{s,x}_u,\omega)|^p\diffns u + |\sigma\cdot B_t|^p\right\}\int_\mathbb{R}w(x)\diffns x \right)\\
		&\quad\le C_{p,T}\left(\int_\mathbb{R}|x|^pw(x)\diffns x + \int_0^t\int_\mathbb{R} |X^{s,x}_u|^pw(x)\diffns x\diffns u + \left\{CTM_2^p + \sup_{t\in [0,T]}|\sigma\cdot B_t|^p\right\}\int_\mathbb{R}w(x)\diffns x \right).
	\end{align*}
	By Gronwall's inequality, this implies
	\begin{align*}
		\int_\mathbb{R}|X^{s,x}_t|^pw(x)\,dx \le C_{p,T}\left(\int_\mathbb{R}|x|^pw(x)\diffns x + \left\{T+CTM_2^p + \sup_{t\in [0,T]}|\sigma\cdot B_t|^p\right\}\int_\mathbb{R}w(x)\diffns x \right)e^{T^2}.
	\end{align*}
	Thus, it follows by H\"older's inequality that
	\begin{align*}
		&E\left[\left(\int_\mathbb{R}|X^{s,x}_t|^pw(x)\diffns x \right)^{1/p}\right]	\le E\left[\int_\mathbb{R}|X^{s,x}_t|^pw(x)\diffns x \right]^{1/p}\\
		 &\quad \le C_{p,T}\left(\int_\mathbb{R}|x|^pw(x)\diffns x + \left\{T+CTE\left[M_2^p\right] + |\sigma^p|E\left[\sup_{t\in [0,T]}| B_t|^p\right]\right\}\int_\mathbb{R}w(x)\diffns x \right)^{1/p}<\infty.
	\end{align*}
This completes the proof since by \eqref{eq:expo moment b2} $M_2$ has exponential moments.
\end{proof}

\begin{appendix}
\section{Compactness criteria} %{Appendix}

The suggested construction of the strong solution for the SDE \eqref{eqmain1r} is based on the subsequent relative compactness criteria from Malliavin calculus (see \cite{DPMN92}.)
 	%[2].
 	 	
\begin{theorem}
 	\label{MCompactness}Let $\left\{ \left( \Omega ,\mathcal{A},P\right) ;H\right\} $ be a Gaussian probability space, that is $\left( \Omega ,%
 	\mathcal{A},P\right) $ is a probability space and $H$ a separable closed subspace of Gaussian random variables in $L^{2}(\Omega )$, which generate the $\sigma $-field $\mathcal{A}$. Denote by $\mathbf{D}$ the derivative operator acting on elementary smooth random variables in the sense that%
 	\begin{equation*}
 		\mathbf{D}(f(h_{1},\ldots,h_{n}))=\sum_{i=1}^{n}\partial
 		_{i}f(h_{1},\ldots,h_{n})h_{i},\text{ }h_{i}\in H,f\in C_{b}^{\infty }(\mathbb{R%
 		}^{n}).
 	\end{equation*}%
 	Further let $\mathbf{D}_{1,2}$ be the closure of the family of elementary smooth random variables with respect to the norm%
 	\begin{equation*}
 		\left\Vert F\right\Vert _{1,2}:=\left\Vert F\right\Vert _{L^{2}(\Omega)}+\left\Vert \mathbf{D}F\right\Vert _{L^{2}(\Omega ;H)}.
 	\end{equation*}%
	Assume that $C$ is a self-adjoint compact operator on $H$ with dense image.
	Then for any $c>0$ the set%
	\begin{equation*}
		\mathcal{G}=\left\{ G\in \mathbf{D}_{1,2}:\left\Vert G\right\Vert_{L^{2}(\Omega )}+\left\Vert C^{-1} \mathbf{D} \,G\right\Vert _{L^{2}(\Omega ;H)}\leq c\right\}
 	\end{equation*}%
 	is relatively compact in $L^{2}(\Omega )$.
\end{theorem}
 	
The relative compactness criteria in this paper required the following result (see \cite[Lemma 1] {DPMN92}).
 	
\begin{lemma}
 	\label{DaPMN} Let $v_{s},s\geq 0$ be the Haar basis of $L^{2}([0,1])$. For any $0<\alpha <1/2$ define the operator $A_{\alpha }$ on $L^{2}([0,1])$ by%
 	\begin{equation*}
 		A_{\alpha }v_{s}=2^{k\alpha }v_{s}\text{, if }s=2^{k}+j\text{ }
 	\end{equation*}%
 	for $k\geq 0,0\leq j\leq 2^{k}$ and%
 	\begin{equation*}
 		A_{\alpha }1=1.
 	\end{equation*}%
 	Then for all $\beta $ with $\alpha <\beta <(1/2),$ there exists a constant $%
 	c_{1}$ such that%
 	\begin{equation*}
 		\left\Vert A_{\alpha }f\right\Vert \leq c_{1}\Big\{ \left\Vert f\right\Vert_{L^{2}([0,1])}+\Big( \int_{0}^{1}\int_{0}^{1}\frac{\left\vert f(t)-f(t^{\prime })\right\vert ^{2}}{\left\vert t-t^{\prime }\right\vert^{1+2\beta }}\diffns t\diffns t^{\prime }\Big) ^{1/2}\Big\}.
 	\end{equation*}
\end{lemma}
The next compactness criteria comes from Theorem \ref{MCompactness} and Lemma \ref{DaPMN}.

\begin{corollary} \label{compactcrit}
	Let $\{X_n\}_{n\geq 1}\in \mathbf{D}_{1,2}$, be a sequence of $\mathcal{F}_1$-measurable random variables such that there exist constants $\alpha > 0$ and $C>0$ with
 	$$
 		\sup_n E \left[ | D_t X_n - D_{t'} X_n |^2 \right] \leq C |t -t'|^{\alpha}
 	$$
 	for $0 \leq t' \leq t \leq 1$, and
 	$$
 		\sup_n\sup_{0 \leq t \leq 1} E \left[ | D_t X_n |^2 \right] \leq C \,.
 	$$
 	Then the sequence $\{X_n\}_{n\geq 1}$, is relatively compact in $L^{2}(\Omega )$.
 	\end{corollary}

	\section{An auxiliary result}

	The following key result generalises \cite[Proposition 2.2.]{Da07} to the case of function with spatial polynomial growth.
	
	\begin{proposition} \label{propmainEstimate}
		Let $B$ be a $d$-dimensional Brownian motion starting from the origin and $ b: [0,1] \times \mathbb{R} \rightarrow \mathbb{R}$ a compactly supported smooth function such that $\|\tilde{b}(t,z)\|\leq k$ with $\tilde{b}(t,z):= \frac{b(t,z)}{1+|z|}, \, (t,z) \in [0,1] \times \mathbb{R}$. Set $B_\sigma(t)=\sum_{i=1}^d\sigma_iB_i(t)\simeq N(0,t\|\sigma\|^2)$ for $\sigma\in \mathbb{R}^d$. Then there exists a constant $C$ depending on $ \sigma$ such that for any even positive number $n$, we have
		\begin{equation}
		\label{propestimate1}
		E \left[\left( \int_{t_0}^t b^\prime(t,x+B_\sigma(t))  \diffns t\right)^n \right]  \leq C^n_{\sigma,k}(1+|x|^{n})(\frac{n}{2})! (t-t_0)^{n/2}.
		\end{equation}
	\end{proposition}
	
	\begin{proof}
		As in \cite{Da07}, the proof is split into several parts.
		
		Let $P_\sigma(t,z) = (2 \pi t\|\sigma\|^2)^{-1/2} e^{-|z|^2/2t\|\sigma\|^2}$ be the Gaussian kernel, then using the joint distribution of $B(t_1),\ldots,B(t_n)$, the left hand side of \eqref{propestimate1} can be written as
		$$
		n!   \int_{t_0 < t_1 < \dots < t_n < t} \int_{\mathbb{R}^{n}} \prod_{i=1}^n   b^\prime(t_i,x+z_i) P_\sigma(t_i - t_{i-1}, z_i - z_{i-1} ) \diffns z_1 \dots \diffns z_n \diffns t_1 \dots \diffns t_n    \,.
		$$
		Define
		$$
		J_n (t_0,t, x,z_0) := \int_{ t_0 < t_1 < \dots < t_n < t} \int_{\mathbb{R}^{n}} \prod_{i=1}^n  b^\prime(t_i,x+z_i) P_\sigma(t_i - t_{i-1}, z_i - z_{i-1} ) \diffns z_1 \dots \diffns z_n \diffns t_1 \dots \diffns t_n,
		$$
		The proposition will be proved if we show that $$|J_n(t_0,t,0)| \leq C_p^n(t-t_0)^{n/2} (1+|x|^n)/ \Gamma( n/2 + 1).$$
Note that the above comes from Proposition 4.10. in \citet{MeMo17}. The result then follows.
	\end{proof}
\end{appendix}

% \bibliographystyle{abbrvnat}
% \bibliography{Biblio}

\vspace{1cm}

\small{  \noindent \textbf{Olivier Menoukeu-Pamen}: University of Liverpool Institute for Financial and Actuarial Mathematics, Department of Mathematical Sciences,
L69 7ZL, United Kingdom and African Institute for Mathematical Sciences, Ghana.\\
\small{\textit{E-mail address:} menoukeu@liverpool.ac.uk\\
Financial support from the Alexander von
Humboldt Foundation, under the programme financed by the German Federal Ministry of Education and Research
entitled German Research Chair No 01DG15010 is gratefully acknowledged.}
  \vspace{.2cm}

  \noindent \textbf{Ludovic Tangpi}: Department of Operations Research and Financial Engineering, Princeton University. 08544, Princeton, NJ, USA 
  {\small\textit{E-mail address:} ludovic.tangpi@princeton.edu}\\
}
\end{document}